\documentclass[11pt]{amsart}
\usepackage{amsmath, amsthm, amssymb, amsfonts}
\usepackage[normalem]{ulem}
\usepackage{hyperref}

\usepackage{xypic, mathtools}
\usepackage{verbatim} 
\usepackage{longtable} 

\usepackage{mathtools}

\theoremstyle{plain}
\newtheorem{thm}{Theorem}
\newtheorem{cor}{Corollary}
\newtheorem{lemma}{Lemma}
\newtheorem{prop}{Proposition}

\theoremstyle{definition}

\theoremstyle{remark}

\usepackage{tikz-cd}
\usetikzlibrary{arrows}

\newcommand{\BC}{{\mathbb{C}}}

\newcommand{\BL}{{\mathbb{L}}}

\newcommand{\BQ}{{\mathbb{Q}}}
\newcommand{\BR}{{\mathbb{R}}}

\newcommand{\BZ}{{\mathbb{Z}}}

\newcommand{\CA}{{\mathcal A}}

\newcommand{\CE}{{\mathcal E}}
\newcommand{\CF}{{\mathcal F}}

\newcommand{\CH}{{\mathcal H}}

\newcommand{\CK}{{\mathcal K}}
\newcommand{\CL}{{\mathcal L}}
\newcommand{\CM}{{\mathcal M}}

\newcommand{\CO}{{\mathcal O}}
\newcommand{\CP}{{\mathcal P}}

\newcommand{\CT}{{\mathcal T}}

\DeclareMathOperator{\Hilb}{Hilb}

\DeclareMathOperator{\id}{id}
\newcommand{\Spec}{\mathop{\rm Spec}\nolimits}

\newcommand{\Pic}{\mathop{\rm Pic}\nolimits}

\newcommand{\p}{\mathbb{P}}

\newcommand{\pt}{{\mathsf{p}}}

\newcommand{\ch}{{\mathrm{ch}}}

\newcommand{\tr}{\mathrm{tr}}

\newcommand{\sr}{\mathsf{sr}}
\newcommand{\Ered}{E_{\text{red}}}
\newcommand{\Chow}{\mathrm{Chow}}
\newcommand{\Coh}{\mathrm{Coh}}
\newcommand{\Ex}{{\curly E\!x}}

\DeclareFontFamily{OT1}{rsfs}{}
\DeclareFontShape{OT1}{rsfs}{n}{it}{<-> rsfs10}{}
\DeclareMathAlphabet{\curly}{OT1}{rsfs}{n}{it}
\renewcommand\hom{\curly H\!om}
\newcommand\ext{\curly Ext}
\newcommand\Ext{\operatorname{Ext}}
\newcommand\Hom{\operatorname{Hom}}
\newcommand\At{\operatorname{At}}
\newcommand\II{\mathbb I}
\newcommand*\dd{\mathop{}\!\mathrm{d}}
\newcommand{\Cone}{\mathop{\rm Cone}\nolimits}

\usepackage{tikz}

\begin{document}
\baselineskip=14.5pt
\author{Georg Oberdieck}
\title{On reduced stable pair invariants}
\address {MIT, Department of Mathematics}
\email{georgo@mit.edu}
\date{\today}
\maketitle

\begin{abstract}
Let $X = S \times E$ be the product of a K3 surface $S$ and an elliptic curve $E$.
Reduced stable pair invariants of $X$ can be defined
via (1) cutting down the reduced virtual class with incidence conditions or
(2) the Behrend function weighted Euler characteristic
of the quotient of the moduli space by the translation action of $E$.
We show that (2) arises naturally as the degree of a virtual class,
and that the invariants (1) and (2) agree.
This has applications to deformation invariance, rationality and a DT/PT correspondence for reduced invariants of $S \times E$.
\end{abstract}

\setcounter{tocdepth}{1} 
\tableofcontents
\setcounter{section}{-1}

\section{Introduction}
\subsection{Pandharipande--Thomas theory}
Let $X$ be a smooth projective threefold
which is Calabi--Yau, i.e. has canonical class $K_X=0$.
The moduli space $P_n(X,\beta)$ parametrizing stable pairs\footnote{
A stable pair $(F,s)$ on $X$ is a pure $1$-dimensional sheaf $F$
on $X$ together with a section $s \in H^0(X,F)$ with $0$-dimensional co-kernel \cite{PT1}.}
$(F,s)$ on $X$ of numerical type
\[
\chi(F)=n\in \mathbb{Z},
\quad [C]= \beta \in H_2(X,\mathbb{Z})
\]
carries a $0$-dimensional virtual
class $[ P_n(X, \beta) ]^{\text{vir}}$
obtained from the perfect obstruction theory of stable pairs \cite{PT1, HL}.
By a result of Behrend \cite{B} its degree can be computed as the topological Euler characteristic of the moduli space
weighted by the Behrend constructible function $\nu : P_n(X, \beta) \to \BZ$,
\begin{equation} \label{dfgsdgsdfF}
\int_{[ P_n(X, \beta)]^{\text{vir}}} 1\  = \ \sum_{k \in \BZ} k \cdot e\left( \nu^{-1}(k) \right).
\end{equation}
The value of \eqref{dfgsdgsdfF}
is the \emph{Pandharipande--Thomas (PT) invariant} $\mathsf{N}_{n,\beta}$ and
encodes information about the number and type of algebraic curves on $X$.

Both sides of \eqref{dfgsdgsdfF} are useful in proving properties of PT invariants.
Via the left hand side (the virtual class side) the $\mathsf{N}_{n,\beta}$
are shown to be invariant under deformations of $X$.
It is also used in the proof of the correspondence between Pandharipande--Thomas and Gromov--Witten theory \cite{PaPix2}.
The right hand side (the Euler characterstic side)
makes PT invariants accessible to cut-and-paste techniques and wall-crossing in the motivic Hall algebra.
This was used by Bridgeland and Toda to prove a rationality result and the correspondence to Donaldson--Thomas (DT) theory \cite{Br1, T16}.

Let $S$ be a nonsingular projective $K3$ surface, and let
$E$ be a nonsingular elliptic curve. Consider the Calabi-Yau threefold
$$X=S \times E \,.$$
Let $\beta\in \text{Pic}(S) \subset H_2(S,\mathbb{Z})$ be a non-zero curve class
and let $d$ be a non-negative integer.
The pair $(\beta,d)$
determines a class in
$H_2(X,\mathbb{Z})$ by
\[ (\beta,d) = \iota_{S \ast}( \beta ) + \iota_{E \ast}(d [E] ) \]
where $\iota_S : S \hookrightarrow X$ and $\iota_E : E \hookrightarrow X$
are inclusions of fibers of the projections to $E$ and $S$ respectively.
%
We are interested in the Pandharipande--Thomas theory of $X$.
However, since $S$ is holomorphic symplectic, the virtual fundamental class vanishes:
\[ [ P_n(X, (\beta,d)) ]^{\text{vir}} = 0. \]
The left hand side of \eqref{dfgsdgsdfF} is zero.
Similarly, the elliptic curve $E$ acts on $X$ by translation in the second factor
and hence on the moduli space $P_n(X,(\beta,d))$ with $1$-dimensional orbits.
Since the Euler characteristic of $E$ vanishes
we conclude also the right hand side of \eqref{dfgsdgsdfF} is zero.
The Pandharipande--Thomas theory of $X$ is trivial\footnote{We ignore here the case $\beta = 0$}.

A non-trivial \emph{reduced} Pandharipande--Thomas theory of $X$ case can be defined in two different ways.

\vspace{5pt}
\noindent (i) Reducing the perfect-obstruction theory by the semi-regularity map
 yields a reduced $1$-dimension virtual cycle
\[ [ P_n(X,(\beta,d)) ]^{\text{red}}. \]
Let $\omega \in H^2(E, \BZ)$ be the class of a point and let
$\beta^{\vee} \in H^2(S, \BQ)$ be a class satisfying $\langle \beta, \beta^{\vee} \rangle = 1$.
\emph{Incidence Pandharipande--Thomas invariants} are defined by
\[
\widetilde{\mathsf{N}}^{X}_{n, (\beta,d)} = 
\int_{[P_{n}(X,(\beta,d))]^{\text{red}}} \tau_0( \pi_1^{\ast}(\beta^{\vee}) \cup \pi_2^{\ast}( \omega ) ) \,,
\]
where the insertion operator $\tau_0( \cdot )$ is defined in \cite{PT1},
and we let $\pi_i$ denote the projection from $S \times E$ to the $i$th factor.

The invariants $\widetilde{\mathsf{N}}^{X}_{n, (\beta,d)}$ are invariant under deformations of $X$ for which $(\beta,d)$ remains of Hodge type $(2,2)$.
If $\beta$ is primitive, a correspondence to the reduced Gromov--Witten theory of $X$ was established in \cite{K3xE}.

\vspace{5pt}
\noindent (ii) The $E$-action on the moduli space $P_n(X, (\beta,d))$ by translation
has finite stabilizer. The quotient stack
\[ P_n(X, (\beta,d)) \, / \, E \]
is therefore Deligne--Mumford.
Following a proposal of J.~Bryan \cite{Bryan-K3xE}
we define \emph{quotient Pandharipande--Thomas invariants} of $X$
as the topological Euler characteristic $e( \cdot )$ of the quotient (taken in the orbifold sense)
weighted by the Behrend function $\nu : P_n(X, (\beta,d)) / E \to \BZ$,
\begin{align*}
\mathsf{N}_{n,(\beta,d)}^{X} & \, = \,
\int_{P_n(X, (\beta,d)) / E} \nu \, \dd{e} \\
& = \, \sum _{k\in \BZ }k\cdot e\big(\nu ^{-1} (k) \big) \,.
\end{align*}

\vspace{7pt}
The main result of the paper is the following comparision
of the invariants defined in (i) and (ii) above.

\begin{thm} \label{4fijdsdf} The quotient and incidence Pandharipande--Thomas invariants of $X = S \times E$ agree:
\[
\mathsf{N}_{n, (\beta,d)}^{X} = \widetilde{\mathsf{N}}^{X}_{n, (\beta,d)}.
 \]
\end{thm}
\vspace{7pt}

We call $\mathsf{N}_{n, (\beta,d)}^{X}$ the reducd Pandharipande--Thomas
or stable pair invariants of $X$.
Theorem~\ref{4fijdsdf} is the following analog of \eqref{dfgsdgsdfF} for the reduced theory:
\[
\int_{[P_{n}(X,(\beta,d))]^{\text{red}}} \tau_0( \pi_1^{\ast}(\beta^{\vee}) \cup \pi_2^{\ast}( \omega ) )
\ = \ 
\sum _{k\in \BZ }k\cdot e\big(\nu ^{-1} (k) \big),
\]
where $\nu : P_n(X, (\beta,d)) / E \to \BZ$ is the Behrend function.

The following was first conjectured by J.~Bryan in \cite{Bryan-K3xE}.
\begin{cor} \label{Corollary_Deformation_invariance}
The invariants $\mathsf{N}_{n, (\beta,d)}^{X}$ are invariant under deformations
of $X$ for which $(\beta,d)$ remains of Hodge type $(2,2)$.
\end{cor}
\begin{proof}
By Theorem \ref{Theorem_Comparision_to_reduced_class} and the deformation invariance of $\widetilde{\mathsf{N}}^{X}_{n, (\beta,d)}$.
\end{proof}

Below we discuss further applications of Theorem~\ref{4fijdsdf} to properties of reduced PT invariants.
The results provide the foundations for a study of the reduced Pandharipande--Thomas theory of $S \times E$ in the future.

It is a difficult problem to explicitly determine PT invariants of Calabi--Yau threefolds.
For example, there is no strict\footnote{A strict Calabi--Yau threefold $X$ also satisfies the condition $H^1(X,\CO_X)=0$.
The PT invariants in the non-strict case often vanish.} Calabi--Yau threefold
for which explicit formulas for PT invariants are conjectured in all curve classes.
In contrast, the reduced PT theory of the product $X = S \times E$ provides an interesting non-trivial example
which is much more accessible\footnote{The reduced theory of an abelian threefold
developed by Gulbrandsen \cite{Gul} is also very interesting.
A conjecture for the invariants in all curve classes can be found in \cite{BOPY}.}.
A complete evaluation of the reduced PT invariants
was conjectured in \cite{K3xE} motivated by physics \cite{KKV} and the calculations \cite{HilbK3}.
If the class on the K3 surface is taken to be primitive the answer is particularly beautiful.
Let $\beta_h \in H_2(S, \BZ)$ be primitive with self-intersection $\langle \beta_h, \beta_h \rangle = 2h-2$.
Then in \cite{K3xE} the following conjecture was made:
\begin{equation} \sum_{d = 0}^{\infty} \sum_{h = 0}^{\infty} \sum_{n \in \BZ} \mathsf{N}_{n, (\beta_h,d)}^{X} y^n q^{h-1} \widetilde{q}^{d-1} = \frac{1}{\chi_{10}(y, q, \widetilde{q})} \label{CONJ} \end{equation}
where $\chi_{10}$ is the Igusa cusp form, a Siegel modular form.
In \cite{K3xP1} we give a proof of \eqref{CONJ} in cases $d \in \{ 1,2 \}$
using an application of Theorem~\ref{4fijdsdf} at a crucial step.
This provided the original motivation for the paper\footnote{
Since the initial draft of the paper a proof of \eqref{CONJ} has appeared in \cite{OS1, OPix1}.
}.

\subsection{Virtual classes}
We describe the results of the paper in more detail.
In Section \ref{Section_Isotrivial_fibrations}, drawing upon ideas of Gulbrandsen \cite{Gul},
we construct isotrivial fibrations
\begin{equation} \label{ISOFIB} p : P_n(X, (\beta,d)) \to E \end{equation}
which are $E$-equivariant with respect to the translation action on the source
and a non-zero multiple of the addition action on the target.
Let $p$ be such a fibration and let $\CK$ be the fiber of $p$ over the zero $0_E \in E$.
We have an isomorphism of quotient stacks
\[ P_n(X, (\beta,d)) / E\ =\ \CK / G \]
where $G$ is a finite group.
In Section \ref{Section_Symmetric_Obstruction_Theory} we construct a symmetric perfect obstruction theory on $\CK$.
Let $[ \CK ]^{\text{vir}}$ be the associated virtual class on $\CK$ and let $\pi : \CK \to P_n(X, (\beta,d))/E$ be the projection.
We define a virtual class on the quotient by
\begin{equation} [ P_n(X, (\beta,d)) / E ]^{\text{vir}} \, := \, \frac{1}{|G|}\, \pi_{\ast} [\, \CK \, ]^{\text{vir}}. \label{134134} \end{equation}
Using a description of the virtual class in terms of the Fulton Chern class,
we show the class \eqref{134134} is independent of the choice of fibration $p$.

Since $\CK$ carries a symmetric obstruction theory and $\pi$ is etale,
Behrend's result \cite{B} implies the degree of $[ P_n(X, (\beta,d)) / E ]^{\text{vir}}$
is the quotient stable pair invariant of $X$ defined in (ii) above,
\[ \mathsf{N}_{n, (\beta,d)}^{X} = \int_{[ P_n(X, (\beta,d)) / E ]^{\text{vir}}} 1. \]

We prove the following comparision result which implies Theorem~\ref{4fijdsdf}. 

\begin{thm} \label{Theorem_Comparision_to_reduced_class} Let $P = P_{n}(X,(\beta,d))$
and let $\pi : P \to P/E$ be the projection. Then
\[ \pi^{\ast} [ P/E ]^{\text{vir}} \, =\, [\, P \, ]^{\text{red}}. \]
\end{thm}

\vspace{4pt}
\noindent \textbf{Remark.} In \cite{Gul} Gulbrandsen constructs isotrivial fibrations on the moduli space of stable pairs
on abelian threefolds if $n \neq 0$. Our construction here works for all Euler characteristics (assuming only $\beta \neq 0$)
thanks to extra line bundles on $X$ coming from the K3 surface.
The construction of the symmetric obstruction theory on the fiber $\CK$ does not follow Gulbrandsen's
argument.
Instead we rely upon the existence of the reduced perfect obstruction theory on the moduli space (proven by Kool-Thomas \cite{KT}).

\vspace{4pt}
\noindent \textbf{Remark.}
It would be interesting to extend our results in two directions:
\begin{enumerate}
 \item[(i)] Construct a symmetric perfect obstruction theory on $P_n(X, (\beta,d))/E$ directly,
 avoiding the detour via the Kummer fiber $\CK$.
 \item[(ii)] The moduli space $P_n(X, (\beta,d))$ is the truncation of a derived scheme
 with $(-1)$-shifted symplectic structure \cite{PTVV}. Can we lift the morphism $p : P_n(X, (\beta,d)) \to E$
 to a morphism of derived schemes such that the fiber naturally inherits the structure of a derived scheme
 with $(-1)$-shifted symplectic form? 
\end{enumerate}

%

\subsection{Applications} \label{Subsection_applications}
Let $\Hilb_{X}((\beta,d), n)$ be the Hilbert scheme parametrizing $1$-dimensional subschemes
$Z \subset X$ with class $[Z] = (\beta,d)$ and Euler characteristic $\chi(\CO_Z) = n$.
Reduced Donaldson-Thomas of $X$ are defined by the Behrend function weighted Euler characteristic
\[ \mathsf{DT}_{n, (\beta,d)} = \int_{\Hilb_{X}((\beta,d), n) / E} \nu \dd{e}  \]
where the quotient is taken with respect to the translation action of $E$.
As above one can show that $\mathsf{DT}_{n, (\beta,d)}$
equals invariants defined by integrating the reduced virtual class of the moduli space against insertions.

Define generating series
\[
\mathsf{DT}_{(\beta,d)}(q) = \sum_n \mathsf{DT}_{n, (\beta,d)} q^n
\quad \text{ and } \quad
\mathsf{PT}_{(\beta,d)}(q) = \sum_n \mathsf{N}_{n, (\beta,d)}^{X} q^n 
\]
of reduced Donaldson-Thomas and stable pair invariants of $X$
in class $(\beta,d)$ respectively.
By a modification of an argument by Bridgeland \cite{Br1} we obtain the following.

\begin{thm} \label{Theorem_applications} Let $\beta \in H_2(S,\BZ)$ be a non-zero curve class.
\begin{enumerate}
 \item[(i)] $\mathsf{DT}_{(\beta,d)}(q) = \mathsf{PT}_{(\beta,d)}(q)$
 \item[(ii)] The series $\mathsf{PT}_{(\beta,d)}(q)$ is the Laurent expansion of a rational function in $q$,
 which is invariant under the transformation $q \mapsto q^{-1}$.
 \item[(iii)] Let $\beta \in H_2(S, \BZ)$ be a primitive class. Then there exist $N \geq 0$ and $\mathsf{n}_{g, (\beta,d)} \in \BZ$ such that
 \[ \mathsf{PT}_{(\beta,d)}(q) = \sum_{g=0}^{N} \mathsf{n}_{g, (\beta,d)} (q^{1/2} + q^{-1/2})^{2g-2}. \]
\end{enumerate}
\end{thm}

Part (i) is the analog of the usual DT/PT correspondence
in the reduced case.
Here there is no contribution from degree $0$
Donaldson-Thomas invariants since $X$ has topological Euler characteristic zero.
Parts (ii-iii) are corollaries of a modification of Toda's equation (Section \ref{Section_Application}).
Part (iii) is a regularity result and used in \cite{K3xP1} at a crucial step.

\subsection{Acknowledgements}
I would like to thank both Davesh Maulik and Yukinobu Toda for
interesting discussions and advice on technical details,
and the latter for giving an inspiring lecture series at MIT in the spring of 2016.
I am also very grateful to Jim Bryan, Michael McBreen, Rahul Pandharipande,
Junliang Shen, and Qizheng Yin for useful discussions.

\section{Preliminaries}
Let $P = P_n(X, (\beta,d))$ denote the moduli space of stable pairs on the threefold $X = S \times E$,
and let $\II = [ \CO_{X \times P} \to \mathbb{F} ]$
be the universal stable pair on $X \times P$,
considered as a complex of coherent sheaves with $\CO_{X \times P}$ in degree $0$ and $\mathbb{F}$ in degree $1$.
The complex $\II$ naturally corresponds to an object in the bounded derived category $D^b(X \times P)$ of coherent sheaves on $X$.

\subsection{The elliptic curve action} \label{Section_Elliptic_Curve_Action}
Consider the action of $E$ on $X = S \times E$ by translation in the second factor
\[ m_X : E \times X \to X, \ \ (x, (s,e)) \mapsto (s, e + x) \,.  \]
Translation by an element $x \in E$ is denoted by
\[ t_x = m_X(x,\, \cdot\, ) : X \to X, (s,e) \mapsto (s, e+x) \,. \]

Let $\iota : E \to E, x \mapsto -x$ be the inverse map.
The pullback $\psi^{\ast}(\II)$ of the universal stable pair $\II$ by the composition
\[ \psi:
E \times X \times P \xrightarrow{\iota \times \id_{X \times P}} E \times X \times P
 \xrightarrow{m_X \times \id_P}  X \times P
\]
defines a family of stable pairs on $X$ over $E \times P$,
and hence, by the universal property of $P$, corresponds to a map
\[ m_{P} \colon E \times P \to P \]
such that we have the isomorphism of complexes\footnote{Here \eqref{34134} is an actual map of complexes, not only a morphism in the derived category.}
\footnote{We let $m_P$ denote here also the map
$E \times X \times P \to X \times P$ that acts on $E \times P$ as the group action $m_P$ and by the identity on $X$.}
\begin{equation} \psi^{\ast}(\II) \cong m_P^{\ast}(\II) \label{34134} \end{equation}
Since $m_X$ is a group action, $m_P$ defines a group action on $P$.
Points $x \in E$ act on elements $I \in P$ by
\[ I + x := m_P( x, I ) = t_{-x}^{\ast}( I ) = t_{x \ast}(I). \]

Consider the diagonal action
\[ m_{X \times P} :
E \times X \times P
\to
X \times P,\ \ (e,x,I) \mapsto (x+e, I+e).
\]
From \eqref{34134} we obtain the isomorphism of complexes $m_{X \times P}^{\ast}( \II ) = \pi_E^{\ast}(\II)$
which satisfies the cocycle conditions
for a descent datum with respect to the quotient map
\[ \rho : X \times P \to (X \times P)/E \]
where the quotient is taken with respect to the diagonal action $m_{X \times P}$.
Hence $\II$ descends along $\rho$:
there exist a complex $\overline{\II}$ on $(P \times X)/E$ such that $\rho^{\ast}(\overline{\II}) \cong \II$.

\subsection{The semiregularity map}
We follow \cite[Section 3]{MPT} and \cite{KT}.
Let
\[ X \xleftarrow{\pi_X} X \times P \xrightarrow{\pi_P} P \]
be the projection maps, and let
\[ \At(\II) \in \Ext^1_{X \times P}( \II,\, \II \otimes \BL_{X \times P} ) \]
be the Atiyah class \cite{HT}.

Let $\At_P(\II)$ be the image of $\At(\II)$ under the composition
\begin{equation*}
\begin{aligned}
\Ext_{X \times P}^1( \II, \II \otimes \BL_{X \times P})
& \overset{(i)}{\to} \Ext_{X \times P}^1( R \hom(\II, \II), \pi_{P}^{\ast} \BL_{P}) \\
& \overset{(ii)}{\to} \Ext_{X \times P}^1( R \hom(\II, \II)_0, \pi_{P}^{\ast} \BL_{P}) \\
& \overset{(iii)}{\cong} \Hom_{P}( R \pi_{\ast} (R \hom(\II, \II)_0 \otimes \omega_P)[2], \BL_P ),
\end{aligned}
\end{equation*}
where (i) arises from the projection
$\BL_{X \times P} \cong \pi_X^{\ast}(\Omega_X) \oplus \pi_P^{\ast}( \BL_P ) \to \pi_P^{\ast} \BL_P$,
(ii) is the restriction to the traceless part $( \ldots  )_0$,
and (iii) is an application of relative Verdier duality.
The morphism $\At_P(\II)$ is the perfect obstruction theory for stable pairs
\[ E^{\bullet} := R \pi_{P \ast}\big( R \hom (\II, \II)_0 \otimes \omega_{\pi_P})[2] \xrightarrow{\At_P(\II)} \BL_P. \]

Let $\At_X(\II)$ be the projection of $\At(\II)$ under
\[
\Ext_{X \times P}^1( \II, \II \otimes \BL_{X \times P}) \to
\Ext_{X \times P}^1( R\hom(\II, \II), \pi_X^{\ast}(\Omega_X) ) \,.
\]
Consider the composition
\begin{equation*}
\begin{aligned}
(E^{\bullet})^{\vee}
\xhookrightarrow{\quad} &\  R \pi_{P \ast} R\hom( \II, \II)[1] \\
\xrightarrow{\cup \At_X(\II)}  & \
R \pi_{P \ast} R\hom( \II, \II \otimes \pi_X^{\ast}(\Omega_X))[2]  \\
\xrightarrow{\ \ \text{tr} \ \ } &\ R \pi_{P \ast}( \pi_X^{\ast}(\Omega_X) )[2] \\
= & \ R \Gamma(X, \Omega_X)[2] \otimes \CO_P.
\end{aligned}
\end{equation*}
Taking $h^1$ we obtain the \emph{semiregularity map}
\begin{equation} \label{SRMAP} \sr \colon (E^{\bullet})^{\vee} \to h^1( (E^{\bullet})^{\vee})[-1] \to H^{1,3}(X) \otimes \CO_P[-1] \end{equation}
which is surjective by \cite[Proposition 11]{MPT}.
Let $E_{\text{red}}^{\bullet}$ be the cone of the morphism $H^{1,3}(X)^{\vee} \otimes \CO_P[1] \to E^{\bullet}$
dual to $\sr$.
We have the diagram
\[
\begin{tikzcd}
E^{\bullet} \ar{d} \ar{r} & E^{\bullet}_{\text{red}} \ar[dotted]{dl}{\exists \text{?}} \\
\BL_P
\end{tikzcd}
\]
and the problem is to find a map $E^{\bullet}_{\text{red}} \to \BL_P$
which is a perfect obstruction theory \cite[Appendix A]{MPT} for $P$.
By arguments of Kool and Thomas \cite{KT} we have the following.

\begin{prop}[\cite{KT}] \label{Proposition_Reduced_Perfect_Obstruction_Theory}
There exist a morphism $E^{\bullet}_{\text{red}} \to \BL_P$ which
is a perfect obstruction theory for $P$.
\end{prop}
\begin{proof}
We sketch the argument which can be found in \cite[Section 2]{KT}.
See \cite{Pridham} for a different approach using derived algebraic geometry.

Let $\mathcal{S} \to B = \Spec( \BC[\epsilon]/\epsilon^2 )$ be an algebraic twistor family of $S$
as in \cite[Section 2.1]{KT}.
The moduli space of stable pairs on the fibers of the family $\mathcal{S} \times E \to B$
is isomorphic to $P$ \cite[Prop 2.3]{KT}.
The perfect obstruction theory relative to the base $B$ is isomorphic to $E^{\bullet}$,
while the associated absolute perfect obstruction theory is quasi-isomorphic to $E^{\bullet}_{\text{red}}$.
\end{proof}

\subsection{Tangent vectors} \label{Section_Tangent_Vectors}
We describe the global vector field on $P$
obtained by moving stable pairs in the $E$-direction on $X$, or equivalently, obtained
by the derivative of the elliptic curve action on $P$.

Let $m : E \times P \to P$ be the action, and consider the composition
\[
T_E \otimes \CO_{E \times P} \hookrightarrow T_{E \times P} \xrightarrow{\dd{m}} m^{\ast} T_P
\]
where $\dd{m}$ is the differential of $m$. Restricting to $0_E \times P \hookrightarrow E \times P$
we obtain the global vector field
\begin{equation} \label{tv1} H^0(T_E) \otimes \CO_P = T_{E, 0_E} \otimes \CO_P \to T_P \cong \ext_{\pi_P}^1(\II, \II)_0, \end{equation}
Since $E$ acts on $P$ with finite stabilizer, \eqref{tv1} is an inclusion.

We will require an alternative description of the map \eqref{tv1}.
Consider the composition
\begin{equation} \label{tv2} 
H^0(T_E) \otimes \CO_P
\cong H^0(X, T_X) \otimes \CO_P
\xrightarrow{ \,\lrcorner \At_X(\II) } \Ext^1(\II, \II) \otimes \CO_P
\to \ext_{\pi_P}^1(\II, \II) \end{equation}
where the first map arises from the identification $H^0(E, T_E) = H^0(X, T_X)$,
and the second map is cup product with $\At_X(\II)$ followed by contraction of $T_X$ with $\Omega_X$.

\begin{lemma} \label{Lemma_Comparision_Tangent_Vectors}
The morphisms \eqref{tv1} and \eqref{tv2} agree (up to sign).
\end{lemma}
\begin{proof}
Consider the diagonal action of $E$ on $X \times P$, and let
\[ \rho : X \times P \to (X \times P)/E\]
be the quotient map. The right hand side is isomorphic to $S \times P$ via the isomorphism
\[ (X \times P)/E \to S \times P,\  [((s,e),I)] \mapsto (s,I-e), \]
and therefore projective. The map $\rho$ is a trivial $E$-bundle.
%
By Section~\ref{Section_Elliptic_Curve_Action}
there is a complex $\overline{\II}$ on $S \times P$ with $\rho^{\ast}(\overline{\II}) = \II$.
By the functoriality of Atiyah classes \cite{HT},
$\At(\II)$ is the pullback of the Atiyah class $\At(\overline{\II})$ on $(X \times P)/E$ under $\rho$.
Hence, $\At(\II): \BL_{X \times P}^{\vee} \to R \hom( \II, \II )[1]$ factors as
\begin{equation} \label{mmggf}
\BL_{X \times P}^{\vee}
\to \rho^{\ast} \BL_{(X \times P)/E}^{\vee}
\xrightarrow{\rho^{\ast} \At(\overline{\II})}
\rho^{\ast} R\hom( \overline{\II}, \overline{\II} )[1]
= R\hom( \II, \II)[1].
\end{equation}

The diagonal action of $E$ on $X \times P$ defines the vector field
\begin{equation}
H^0(E, T_E) \otimes \CO_{X \times P} \to \BL_{X \times P}^{\vee}
=
\pi_X^{\ast}(T_X) \oplus \pi_P^{\ast}(\BL_P^{\vee})
\label{vec2} \end{equation}
which is isomorphic to the direct sum of the pullbacks to $X \times P$ of
\eqref{tv1} and the natural inclusion $H^0(X, T_X) \otimes \CO_X \to T_X$.
By the factorization \eqref{mmggf} and since $\rho$ contracts the orbits of the diagonal action, the composition
\[
H^0(E, T_E) \otimes \CO_{X \times P}
\to
\BL_{X \times P}^{\vee}
\xrightarrow{\At(\II)}
R \hom(\II, \II)[1]
\]
vanishes.
We find that the morphism
\[ H^0(E, T_E) \otimes \CO_{X \times P} \to \pi_P^{\ast}(\BL_P^{\vee}) \xrightarrow{\At_P(\II)} R\hom(\II, \II)[1] \]
and
\[ H^0(E, T_E) \otimes \CO_{X \times P} \to \pi_X^{\ast}(T_X) \xrightarrow{ \,\lrcorner \, \At_X(\II) } R\hom(\II, \II)[1] \]
are the negative of each other. The claim now follows from adjunction.
\end{proof}

Because $X$ has trivial canonical bundle, relative Serre duality provides
the isomorphism
\[
\theta : E^{\bullet} \to (E^{\bullet})^{\vee}[1]
\]
satisfying $\theta^{\vee}[1] = \theta$. The perfect obstruction theory $E^{\bullet}$ is symmetric \cite{B}.

Let $s : H^0(T_E)^{\vee} \otimes \CO_P \to H^{1,3}(X) \otimes \CO_P$ be the isomorphism
induced by
the non-degenerate pairing
\[ H^0(X, T_X) \otimes H^3(X, \Omega_X) \to H^3(X, \CO_X) \cong \BC. \]
and the identification $H^0(E, T_E) = H^0(X, T_X)$.
Let also
\[ \partial^{\vee} : E^{\bullet} \to h^0(E^{\bullet}) = \Omega_P \to H^0(T_E)^{\vee} \otimes \CO_P \]
denote the composition of the natural cut-off map and the dual of \eqref{tv1}.

\begin{lemma} \label{Lemma_Diagram_commutes}
The following diagram commutes:
\[
\begin{tikzcd}
E^{\bullet} \ar{d}{\theta} \ar{r}{\partial^{\vee}} & H^0(T_E)^{\vee} \otimes \CO_P \ar{d}{s} \\
(E^{\bullet})^{\vee}[1] \ar{r}{\sr[1]} & H^{1,3}(X) \otimes \CO_P,
\end{tikzcd}
\]
where $\sr$ is the semiregularity map \eqref{SRMAP}.
\end{lemma}
\begin{proof}
Since $E^{\bullet}$ and $(E^{\bullet})^{\vee}[1]$ are both of amplitude contained in $[-1,0]$, the horizontal maps
factor as
\[
\begin{tikzcd}
E^{\bullet} \ar{d}{\theta} \ar{r}{h^0}
& h^0(E^{\bullet}) = \ext^1_{\pi_P}(\II, \II)^{\vee} \ar{d}{h^0(\theta)} \ar{r}{h^0(\partial^{\vee})}
& H^0(T_E)^{\vee} \otimes \CO_P \ar{d}{s} \\
(E^{\bullet})^{\vee}[1] \ar{r}{h^0} & h^0((E^{\bullet})^{\vee}[1]) = \ext^2_{\pi_P}(\II,\II) \ar{r}{h^1(\sr)} & H^{1,3}(X) \otimes \CO_P.
\end{tikzcd}
\]
The square on the left commutes. We need to show the right square commutes.
By definition, $h^0(\theta)$ is the isomorphism induced by the pairing
\[ \ext^1_{\pi_P}(\II, \II) \times \ext^2_{\pi_P}(\II,\II) \xrightarrow{\cup} \ext^3(\II, \II) \xrightarrow{\tr} H^3(X, \CO_X) \otimes \CO_P. \]
By Lemma~\ref{Lemma_Comparision_Tangent_Vectors} the map $\partial^{\vee}$ is dual to taking the interior product with $\At_X(\II)$.
Hence, the claim now follows just as in \cite[Proof of Lemma 12]{MPT} by an application of the homotopy formula:
With $\partial_t \in H^0(T_X)$ a generator with dual $dt \in H^0(\Omega_X)$ we have
\[ \tr\big( f \cup ( \partial_t \,\lrcorner \At_X(\II))\big) \wedge \dd{t} = \tr( f \cup \At_X(\II) ) \]
for every $f \in \ext^2_{\pi_P}(\II, \II)$.
\end{proof}

\subsection{Symmetric complexes} \label{Section_Symmetric_complexes}
Let
\[ \CT^{\bullet} = \Cone( E^{\bullet} \xrightarrow{\partial^{\vee}} H^0(T_E)^{\vee} \otimes \CO_P )[-1] \]
and recall that the reduced complex $E^{\bullet}_{\text{red}}$
fits into the exact triangle
\[ H^{1,3}(X)^{\vee} \otimes \CO_P \xrightarrow{\sr^{\vee}} E^{\bullet} \to E_{\text{red}}^{\bullet} \,. \]
Since $\Hom(\CT, \CO_P[-1]) = 0$, the square of Lemma~\ref{Lemma_Diagram_commutes} can be uniquely completed
to the morphism of horizontal exact triangles
\begin{equation} \label{EXACTTRIANGLE1}
\begin{tikzcd}
\CT \ar{d}{\widetilde{\theta}} \ar{r} & E^{\bullet} \ar{d}{\theta} \ar{r}{\partial^{\vee}} & H^0(T_E)^{\vee} \otimes \CO_P \ar{d}{s} \\
(E_{\text{red}}^{\bullet})^{\vee}[1] \ar{r} & (E^{\bullet})^{\vee}[1] \ar{r} & H^{1,3}(X) \otimes \CO_P .
\end{tikzcd}
\end{equation}
Since $\theta$ and $s$ are isomorphisms, 
so is the induced map $\widetilde{\theta}$.

Since all maps $\CO_P[1] \to \CO_P$ vanish,
the semiregularity map $\sr^{\vee}$ induces a map
$H^{1,3}(X)^{\vee} \otimes \CO_P[1] \to \CT$.
Let
\[ G^{\bullet} = \Cone(H^{1,3}(X)^{\vee} \otimes \CO_P[1] \to \CT) \]
be the cone. By virtue of the commutative diagram
of vertical and horizontal exact triangles
\begin{equation} \label{BIGSQUARE}
\begin{tikzcd}
H^{1,3}(X)^{\vee} \otimes \CO_P[1] \ar{d}{\sr^{\vee}} \ar{r}{=} & H^{1,3}(X)^{\vee} \otimes \CO_P[1] \ar{d}{\sr^{\vee}} \ar{r} & 0 \ar{d} \\
\CT^{\bullet} \ar{d} \ar{r} & E^{\bullet} \ar{d} \ar{r}{\partial^{\vee}} & H^0(T_E)^{\vee} \otimes \CO_P \ar{d}{=} \\
G^{\bullet} \ar{r} & E^{\bullet}_{\text{red}} \ar{r}{\partial^{\vee}} & H^0(T_E)^{\vee} \otimes \CO_P \,,
\end{tikzcd}
\end{equation}
we also have
\begin{equation} G^{\bullet} = \Cone( E^{\bullet}_{\text{red}} \xrightarrow{\partial^{\vee}} H^0(T_E)^{\vee} \otimes \CO_P )[-1] \,. \label{Gdef2} \end{equation}
Here we let $\partial^{\vee}$ and $\sr^{\vee}$ denote also the induced morphisms.

By an argument identical to the proof of Lemma \ref{Lemma_Diagram_commutes},
the left square in the following diagram of horizontal exact triangles commutes:
\begin{equation} \label{EXACTTRIANGLE2}
 \begin{tikzcd}
  H^{1,3}(X)^{\vee} \otimes \CO_P[1] \ar{d}{s^{\vee}[1]} \ar{r}{\sr^{\vee}} & \CT^{\bullet} \ar{d}{\widetilde{\theta}} \ar{r} & G^{\bullet} \ar[dotted]{d}{\lambda} \\
 H^0(T_E) \otimes \CO_P[1] \ar{r}{\partial} & (E^{\bullet}_{\text{red}})^{\vee}[1] \ar{r} & (G^{\bullet})^{\vee}[1].
 \end{tikzcd}
\end{equation}
Let $\lambda : G^{\bullet} \to (G^{\bullet})^{\vee}[1]$ be the induced morphism
(uniquely defined since $\Hom(\CO_P[2], (G^{\bullet})^{\vee}[1]) = 0$).
As before, $\lambda$ is an isomorphism.


\begin{prop} \label{Proposition_Gbullet_is_symmetric} We have $\lambda^{\vee}[1] = \lambda$. 
Hence $\lambda: G^{\bullet} \to (G^{\bullet})^{\vee}[1]$
is a nondegenerate symmetric bilinear form of degree $1$
in the sense of \cite[Defn. 3.1]{B}.
\end{prop}

\noindent \textbf{Remark.}
By definition, we have 
\[
\widetilde{\theta} = \Cone( \theta \to s )[-1]
\quad \text{and} \quad
\lambda = \Cone( s^{\vee}[1] \to \widetilde{\theta} ) \,.
\]
Dualizing \eqref{EXACTTRIANGLE1} and using $\theta^{\vee}[1] = \theta$ one has
$\widetilde{\theta}^{\vee}[1]
= \Cone( s^{\vee}[1] \to \theta )$, and dualizing \eqref{EXACTTRIANGLE2} yields
$\lambda^{\vee}[1] = \Cone( \widetilde{\theta}^{\vee}[1] \to s )[-1]$.
Since the cones are unique in each case,
we expect taking cones and cocones to commute\footnote{
In general, the cone of a map of morphisms in a triangulated category is not uniquely defined.
A general statement of this form is false.} and hence $\lambda \cong \lambda^{\vee}[1]$.
The following proof makes this a rigorous argument.

\begin{proof}
We write $E, E_{\text{red}}, \CT, G$ for $E^{\bullet}, E^{\bullet}_{\text{red}}, \CT^{\bullet}, G^{\bullet}$ respectively,
and identify $H^0(T_E)^{\vee} = H^{1,3}(X) = \BC$.
Let 
\begin{equation} \label{MFbFBDF}
\begin{tikzcd}
\CT \ar{r}{\iota_{\CT}} \ar{d}{q_G} & E \ar{d}{q_{\text{red}}} \\
G \ar{r}{\iota_G} & E_{\text{red}}.
\end{tikzcd}
\end{equation}
denote the lower left square in the commutative diagram \eqref{BIGSQUARE}.

We recall the definition of $\widetilde{\theta}$ through the diagram \eqref{EXACTTRIANGLE1}.
Since we have $\Hom(\CT, \CO_P[-1]) = 0$, the bottom row of the diagram
\[
\begin{tikzcd}
& \Hom(E, E^{\vee}[1]) \ar{d}{ \cdot\, \circ\, \iota_T } \\
0 \to \Hom(T, \Ered^{\vee}[1]) \ar{r}{ q_{\text{red}}^{\vee}[1]\, \circ\, \cdot} & \Hom(T, E^{\vee}[1])
\ar{r}{ \sr[1] \, \circ \, \cdot } & \Hom(\CT, \CO_P)
\end{tikzcd}
\]
is exact, where we have written $f \circ \cdot$ (resp. $\cdot \circ f$) for the map obtained from composing (resp. precomposing) with a morphism $f$.
By the commutativity of \eqref{EXACTTRIANGLE1} we have
$\sr[1] \circ \theta \circ \iota_\CT = 0$. The morphism $\widetilde{\theta}$ is then defined as the unique element in $\Hom(\CT, \Ered^{\vee}[1])$ such that
\begin{equation} q_{\text{red}}^{\vee}[1] \circ \widetilde{\theta} = \theta \circ \iota_\CT. \label{eqni} \tag{i} \end{equation}

Similarly, from \eqref{EXACTTRIANGLE2} we have the diagram
\[
\begin{tikzcd}
& \Hom(\CT, \Ered^{\vee}[1]) \ar{d}{ \iota_G^{\vee}[1] \, \circ \, \cdot} \\
0 \to \Hom(G, G^{\vee}[1]) \ar{r}{ \cdot \, \circ \, q_G } & \Hom(\CT, G^{\vee}[1])
\ar{r}{ \cdot \, \circ \, \sr^{\vee}} & \Hom(\CO[1], G^{\vee}[1])
\end{tikzcd}
\]
with exact bottom row.
Since the left square in \eqref{EXACTTRIANGLE2} commutes, we have
$\iota_G^{\vee}[1] \circ \widetilde{\theta} \circ \sr^{\vee} = 0$. The morphism $\lambda \in \Hom(G, G^{\bullet}[1])$
is then defined by
\begin{equation} \lambda \circ q_G = \iota_G^{\vee}[1] \circ \widetilde{\theta} . \tag{ii} \label{eqnii} \end{equation}
Dualizing the diagrams \eqref{EXACTTRIANGLE1} and \eqref{EXACTTRIANGLE2} and using $\theta^{\vee}[1] = \theta$
one also has
\begin{align}
\widetilde{\theta}^{\vee}[1] \circ q_{\text{red}} & = \iota_\CT^{\vee}[1] \circ \theta  \label{eqniii} \tag{iii} \\
q_G^{\vee}[1] \circ \lambda^{\vee}[1] & = \widetilde{\theta}^{\vee}[1] \circ \iota_G. \label{eqniv} \tag{iv}
\end{align}

We show $\lambda^{\vee}[1] = \lambda$.
Since $\Hom(\CO_P[2], G^{\vee}[1]) = \Hom(\CO_P[2], \CT^{\vee}[1]) = 0$ and $\Hom(G, \CO_P[-1]) = \Hom( \CT, \CO_P[-1]) = 0$ we have
the diagram of exact rows and columns
\[
\begin{tikzcd}
0 \ar{r} & \Hom(G, \CT^{\vee}[1]) \ar{r}{ \cdot \, \circ \, q_G } & \Hom(\CT, \CT^{\vee}[1]) \\
0 \ar{r} & \Hom(G, G^{\vee}[1]) \ar{u}{q_G^{\vee}[1]\, \circ\, \cdot} \ar{r}{ \cdot \, \circ \, q_G } & \Hom( \CT, G^{\vee}[1]) \ar{u}{q_G^{\vee}[1]\, \circ\, \cdot} \\
& 0 \ar{u} & 0 \ar{u}.
\end{tikzcd}
\]
Hence it is enough to show
$q_{G}^{\vee}[1] \circ \lambda^{\vee}[1] \circ q_G = q_{G}^{\vee}[1] \circ \lambda \circ q_G$ as an element of $\Hom( \CT, \CT^{\vee}[1] )$.
But we have the series of equivalences
\begin{alignat*}{3}
&&  q_{G}^{\vee}[1] \circ \lambda^{\vee}[1] \circ q_G & = q_{G}^{\vee}[1] \circ \lambda \circ q_G \\
\Longleftrightarrow & \quad &
(\widetilde{\theta}^{\vee}[1] \circ \iota_G) \circ q_G
& = q_G^{\vee}[1] \circ ( \iota_G^{\vee}[1] \circ \widetilde{\theta} )
& (\text{by } \eqref{eqnii} \text{ and } \eqref{eqniv}) \\
\Longleftrightarrow & \quad &
\widetilde{\theta}^{\vee}[1] \circ q_{\text{red}} \circ \iota_T
& = \iota_{\CT}^{\vee}[1] \circ q_{\text{red}}^{\vee}[1] \circ \widetilde{\theta}
& \quad (\text{by commutativity of } \eqref{MFbFBDF}) \\
\Longleftrightarrow & \quad &
\iota_\CT^{\vee}[1] \circ \theta \circ \iota_{\CT}
& = \iota_{\CT}^{\vee}[1] \circ \theta \circ \iota_{\CT}
& (\text{by } \eqref{eqni} \text{ and } \eqref{eqniii})
\end{alignat*}
so the proof is complete.
\end{proof}

\begin{lemma} \label{Lemma_Gbullet_perfect} The complex $G^{\bullet}$ is perfect of amplitude contained in $[-1,0]$. \end{lemma}
\begin{proof}
Let
$(E^{\bullet})^{\vee} = [ E_0 \to E_1 ]$
be a presentation of $(E^{\bullet})^{\vee}$ by vector bundles $E_0, E_1$.
The inclusion \eqref{tv1} yields the injection 
\begin{equation} \label{vfinj} \CO_P \to h^0(E_{\bullet}) \to E_0 \end{equation}
and the semiregularity map \eqref{SRMAP} yields the surjection
\begin{equation} \label{obsquotient} E_1 \to E_1 / E_0 = \text{Ob} \to \CO_P \,. \end{equation}
Let $G_0$ be the quotient of \eqref{vfinj} and let $G_1$ be the kernel of \eqref{obsquotient}.
Since $G_1$ is the kernel of a surjective map of vector bundles, it is a vector bundle.
The injection \eqref{vfinj} arises from the infinitesimal action of $E$ on the moduli space.
Since $E$ acts with finite stabilizers, \eqref{vfinj}
is injective at all closed points and defines a
subbundle of $E_0$; hence $G_0$ is a vector bundle.\footnote{Alternatively, \eqref{vfinj}
is dual to the surjection \eqref{obsquotient} via a local selfdual representative $E_0 \to E_1^{\vee}$ of the
non-degenerate form on $E^{\bullet}$ \cite[Section 3]{B}. Hence $G_0$ and $G_1$ are their respective duals.}; 
Since $\CO_P \to E_0 \to E_1$ vanishes, there exist an induced map $G_0 \to E_1$.
Since also $G_0 \to E_1 \to \CO_P$ vanishes we obtain an induced map $G_0 \to G_1$, see the diagram
\[
\begin{tikzcd}
0 \ar{r} & \CO_P \ar{r} & E_0 \ar{r} \ar{d} & G_0 \ar{r} \ar[dotted]{d} & 0 \\
0 & \CO_P \ar{l} & E_1 \ar{l}           & G_1 \ar{l} & 0 \ar{l} \,.
\end{tikzcd}
\]
The complex $[G_0 \to G_1]$ is a presentation of $(G^{\bullet})^{\vee}$ by vector bundles.
\end{proof}

\section{Isotrivial fibrations} \label{Section_Isotrivial_fibrations}
\subsection{Overview}
Let $\beta \in \Pic(S)$ be a curve class, let $d, n \in \BZ$ with $d \geq 0$ and let
$P = P_n(X,(\beta,d))$. In this section we will \emph{not} require $\beta$ to be non-zero.

For all $k \in \BZ$ let  
\[ \sigma_k = + \circ (k \times \id_E) : E \times E \to E,\ \ (x, e) \mapsto e + k x \]
be the action of $E$ on itself by $k$ times the translation. Recall also from Section \ref{Section_Elliptic_Curve_Action} the action
of $E$ on $P$ by translation
\[ m_P : E \times P \to P, \ (x, I) \mapsto I + x = t_{x \ast} I \,. \]

Here we will prove the following result.
\begin{prop} \label{Proposition_Map_p} Let $\mathsf{c} \in \Pic(S)$ be an element
and let $k = \langle \mathsf{c}, \beta \rangle + n$.
There eixst a morphism
\[ p_{\mathsf{c}} \colon P_n(X, (\beta,d)) \to E \]
which is $E$-equivariant with respect to $m_P$ and $\sigma_k$. In particular,
for all $I \in P$ and $x \in E$ we have $p_{\mathsf{c}}(I + x) = p_{\mathsf{c}}(I) + kx$.
\end{prop}

Let $p_{\mathsf{c}}$ be morphism as in Proposition~\ref{Proposition_Map_p}
and assume $k = \langle \mathsf{c}, \beta \rangle + n \neq 0$.
Let $\CK = p_{\mathsf{c}}^{-1}(0_E)$ be the fiber over $0_E \in E$.
Then we have the fiber diagram
\[
\begin{tikzcd}
E \times \CK \ar{d}{\pi_{E}} \ar{r}{m_{\CK}} & P \ar{d}{p_{\mathsf{c}}} \\
E \ar{r}{k} & E
\end{tikzcd}
\]
where $\pi_E$ is the projection to the first factor
and $m_\CK$ is the restriction of $m_P$ to $E \times \CK$.
Hence $p_{\mathsf{c}}$ is an \'etale isotrivial fibration.

\subsection{Construction of $p_{\mathsf{c}}$}
Let $\widehat{E} = {\mathrm Pic}^0(E)$ be the dual of the elliptic curve $(E,0)$,
and let $\CP \to E \times \widehat{E}$
be the Poincar\'e line bundle defined by the conditions:
\begin{enumerate}
 \item[(i)] For all $\xi \in \Pic^0(E)$ the restriction $\CP_{\xi} = \CP|_{E \times \xi}$ is isomorphic to $\xi$,
 \item[(ii)] $\CP|_{0 \times \widehat{E}} \cong \CO_{\widehat{E}}$.
\end{enumerate}
It follows that the restriction $\CP_x = \CP|_{x \times \widehat{E}}$
is isomorphic to $x \in E$ under the identification $\Pic^0(\widehat{E}) \equiv E$.

Let $\widehat{X} = S \times \widehat{E}$ and let
$\CP_X \to X \times \widehat{X}$ denote the pullback of the Poincar\'e
line bundle by the natural projection
$X \times \widehat{X} \to E \times \widehat{E}$.
Define the Fourier-Mukai transform
\begin{equation*}
\Phi_{\CP_X} : D^b(X) \to D^b(\widehat{X}), \  \CE \mapsto Rq_{\ast}( \CP_X \otimes p^{\ast}(\CE))
\label{PhiP}
\end{equation*}
where $X \xleftarrow{p} X \times \widehat{X} \xrightarrow{q} \widehat{X}$ are the respective projections.

Let $\CL$ be a line bundle on $S$ with first Chern class $c_1(\CL) = \mathsf{c}$,
and let $\pi_S : \widehat{X} \to S$ be the projection. Tensoring with $\pi_S^{\ast} \CL$ defined the map
\[ ( \, \cdot \, ) \otimes \pi_S^{\ast} \CL : D^b( \widehat{X} ) \to D^b( \widehat{X} ),\ \CE \mapsto \CE \otimes \pi_S^{\ast} \CL. \]
Finally, let $\pi_{\widehat{E}} : \widehat{X} \to \widehat{E}$ be the projection to $\widehat{E}$.

We define $p_{\mathsf{c}}$ on $\BC$-valued points by the composition
\begin{multline} \label{compoo}
P(\BC) = P_n(X, (\beta,d))(\BC) \ \hookrightarrow \ D^b(X)
\xrightarrow{\Phi_{\CP_X}} D^b(\widehat{X}) \\
\xrightarrow{ (\ \cdot\ ) \otimes \pi_{S}^{\ast} \CL } D^b(\widehat{X})
\xrightarrow{ R \pi_{\widehat{E}\, \ast} } D^b(\widehat{E})
\xrightarrow{ \det } \Pic^m(\widehat{E}) = E
\end{multline}
for some fixed $m$ (determined by $(\beta, d)$ and $n$).
Here the first map is the natural inclusion obtained by considering stable pairs
as elements in the derived category, and the identification $\Pic^m(\widehat{E}) = E$
is obtained by twisting with $m$ times the ideal sheaf of the zero $0_{\widehat{E}} \in \widehat{E}$.

We construct $p_{\mathsf{c}}$ as an algebraic map
by giving a line bundle
on the product $P_n(X, (\beta,d)) \times \widehat{E}$
and applying the universal property of $\Pic^m(\widehat{E})$.
Let $\II$ be the universal stable pair over $X \times P$ and consider the element
\begin{equation} R \pi_{\widehat{E} \times P\, \ast} \left( \Phi_{\CP_{X \times P}}(\II) \otimes \pi_{S}^{\ast} \CL \right) \label{Asdsd} \end{equation}
obtained from applying the first four maps of \eqref{compoo} (considered as maps \emph{relative} to the base $P$) to $\II$.
Since $\II$ is a perfect complex \cite{PT1} and
derived pushforward by smooth projective maps preserves perfectness \cite[Prop. 2.1.10]{HL}
the complex \eqref{Asdsd} is perfect. We let
\[ p_{\mathsf{c}} : P_n(X, (\beta,d)) \to \Pic^m(\widehat{E}) = E \]
be the map induced by the determinant of \eqref{Asdsd}.
By construction, $p_{\mathsf{c}}$ agrees with the previous definition on $\BC$-valued points.

\subsection{Properties of $p_{\mathsf{c}}$}
Let $\pt \in H^4(S ,\BZ)$ and $\omega \in H^2(E,\BZ)$ be the class of a point on $S$ and $E$ respectively.
We will freely use the identification induced by the K\"unneth decomposition
\[ H^{\ast}(X, \BQ) = H^{\ast}(S , \BQ) \otimes H^{\ast}(E, \BQ) \,. \]

Let $\CE \in D^b(X)$ be an element with Chern character
\[ \ch(\CE) = (r, \ell + a \omega, \ell' \omega + d \pt, n \pt \omega) 
\in \oplus_{i=0}^{3} H^{2i}(S \times E, \BQ) \]
for some $r,a,d,n \in \BZ$ and $\ell, \ell \in \Pic(S)$.

\begin{lemma} \label{Psi_on_cohomology}
The complex
$\psi(\CE) = R \pi_{\widehat{E}\, \ast} \left( \Phi_{\CP_{X}}(\II) \otimes \pi_{S}^{\ast} \CL \right) \in D^b(\widehat{E})$
has Chern characters
\begin{align*}
\ch_0( \psi(\CE) ) & = a \big( 2 + \frac{1}{2} \langle \mathsf{c}, \mathsf{c} \rangle \big) + \langle \mathsf{c}, \ell' \rangle + n \\
\ch_1( \psi(\CE) ) & = \Big[ r \left( - 2 - \frac{1}{2} \langle \mathsf{c}, \mathsf{c} \rangle \right) - \langle \ell, \mathsf{c} \rangle - d \Big] \omega
\end{align*}
\end{lemma}
\begin{proof}
The Fourier-Mukai transform $\Phi_{\CP_X}$ acts on cohomology by (signed) Poincare duality
on the $H^{\ast}(E)$ factor, and trivially on the $H^{\ast}(S)$ factor \cite[Lem.9.23]{H}.
Tensoring with $\CL$ acts on cohomology by multiplication with $\exp(c_1(\pi_S^{\ast} \CL))$.
The action of pushforward to $\widehat{E}$ on cohomology can be computed by the Grothendieck-Riemann-Roch formula.
Putting everything together, the claim follows by a direct computation.
\end{proof}

For every $I \in P_n(X, (\beta,d))$ we have $\ch(I) = (1, 0, -\beta \omega - d \pt, -n)$ and
therefore by Lemma \ref{Psi_on_cohomology}
\[ \ch(\psi(I)) = - n - \langle \mathsf{c},  \beta \rangle +  \big( -2 - \langle \mathsf{c}, \mathsf{c} \rangle/2 + d \big) \omega. \]

\begin{lemma} \label{Lemma_intertwining}
For all $\CE \in D^b(X)$ and $x \in E$ we have 
\[ \Phi_{\CP_X}( t_x^{\ast} \CE  ) = \pi_{\widehat{E}}^{\ast} \CP_{-x} \otimes \Phi_{\CP_X}(\CE), \]
where $t_x : X \to X, (s, e)\mapsto (s,e+x)$ is translation by $x$.
\end{lemma}
\begin{proof}

We have
\begin{equation}
\Phi_{\CP_X}(t_x^{\ast}(\CE))
= Rq_{\ast}\big( \CP_X \otimes p^{\ast} t_x^{\ast} \CE  \big)
= Rq_{\ast} \widetilde{t}_x^{\ast} \big( (\widetilde{t}_{-x}^{\ast} \CP_X) \otimes p^{\ast} \CE  \big),
\label{equa35} \end{equation}
where $\widetilde{t}_x = t_x \times \id_{\widehat{X}}$.
We have 
$Rq_{\ast} \tilde{t}_x^{\ast} = Rq_{\ast} (\tilde{t}_{-x})_{\ast} = Rq_{\ast}$
and, by definition of the Poincar\'e bundle, $\tilde{t}_x^{\ast}(\CP_X) = \CP_X \otimes q^{\ast} \pi_{\widehat{E}}^{\ast}( \CP_{-x} )$.
Hence the claim follows from \eqref{equa35} by the projection formula.
\end{proof}

\begin{lemma} The map $p_{\mathsf{c}}$ is $E$-equivariant with respect to $m_P$ and $\sigma_k$ where
$k = \langle \mathsf{c}, \beta \rangle + n$.
In particular, 
$p_{\mathsf{c}}(I + x) = p(I) + kx$
for all $x \in E$ and $I \in P$.
\end{lemma}
\begin{proof} For points $I \in P$ we have
\begin{alignat*}{2}
p_{\mathsf{c}}(I + x)
& = p_{\mathsf{c}}( t_{-x}^{\ast}(I) ) \\
& = \det R \pi_{\widehat{E} \ast} \big( \pi_{\widehat{E}}^{\ast}(\CP_x) \otimes \pi_S^{\ast} \CL \otimes \Phi_{X}(I) \big)
& \quad \quad (\text{by Lemma \ref{Lemma_intertwining}}) \\
& = \det\Big( \CP_x \otimes R \pi_{\widehat{E} \ast} \big( \pi_S^{\ast} \CL \otimes \Phi_{X}(I) \big) 
& \quad \quad (\text{projection formula}) \\
& = \CP_x^{k} \otimes p_{\mathsf{c}}(I) = I + kx
& \quad \quad (\text{by Lemma \ref{Psi_on_cohomology}}).
\end{alignat*}
The general case follows by an identical argument applied to the universal complex $\II$ relative over $P$.
\end{proof}

\section{Virtual classes} \label{Section_Symmetric_Obstruction_Theory}
\subsection{Symmetric obstruction theory}
Let $\beta \in \Pic(S)$ be a \emph{non-zero} curve class, and let $d, n \in \BZ$ with $d \geq 0$.
Let $P = P_n(X,(\beta,d))$
be the moduli space of stable pairs, and let
$p : P \to E$
be a morphism as in Proposition \ref{Proposition_Map_p}
which is $E$-equivariant with respect to a non-zero multiple of the translation action on the target.
In particular, $p$ is an isotrivial fibration. Let
\[ \CK = p^{-1}(0_E) \]
be the fiber over the zero $0_E \in E$. 

Let $m_P : E \times P \to P$ be the translation action and let $E \xleftarrow{\pi_E} E \times P \xrightarrow{\pi_P} P$
denote the projections. Consider the composition
\begin{equation} \label{uyyt} m^{\ast} \BL_P \to \BL_{E \times P} = \pi_E^{\ast}(\Omega_E) \oplus \pi_P^{\ast} \BL_P \to \pi_E^{\ast}(\Omega_E) \end{equation}
of the derivative map and the projection to the first factor.
Restricting \eqref{uyyt} to $0_E \times P \hookrightarrow E \times P$ we obtain the morphism
\[ d : \BL_P \to \Omega_{E,0} \otimes \CO_P \]
which can be identified with the second map in the natural exact triangle
$\pi^{\ast} \BL_{P/E} \to \BL_P \to \BL_{\pi} \cong \CO_P$
obtained from the quotient map $\pi : P \to P/E$.

Let $g : \Ered^{\bullet} \to \BL_P$ be the \emph{reduced} perfect obstruction theory of
Proposition~\ref{Proposition_Reduced_Perfect_Obstruction_Theory},
and let $G^{\bullet}$ be the complex constructed in Section~\ref{Section_Symmetric_complexes}.
By \eqref{Gdef2} we have the diagram of horizontal exact triangles
\begin{equation} \label{DIAG}
\begin{tikzcd}
G^{\bullet} \ar[dotted]{d}{\varphi} \ar{r} & \Ered^{\bullet} \ar{r} \ar{d}{g} & \Omega_{E,0} \otimes \CO_P \ar{d}{=} \\
\pi^{\ast} \BL_{P/E} \ar{r} & \BL_P \ar{r}{d} & \Omega_{E,0} \otimes \CO_P .
\end{tikzcd}
\end{equation}
where we have identified $H^0(T_E)^{\vee} = \Omega_{E,0_E}$.
By Section~\ref{Section_Tangent_Vectors} the right hand square commutes.
Let 
\[ \varphi : G^{\bullet} \to \pi^{\ast} \BL_{P/E} \]
be the induced morphism which is unique since $\Hom(G^{\bullet}, \CO_P[-1]) = 0$.

Let $\iota : \CK \to P$ be the inclusion, and consider the derived restriction
\[ L\iota^{\ast}(\varphi) : L \iota^{\ast} G^{\bullet} \to L \iota^{\ast} \pi^{\ast} \BL_{P/E} = \widetilde{\pi}^{\ast}(\BL_{P/E}) \]
where we let $\widetilde{\pi} : \CK \to P/E$ be the projection.
Since $\widetilde{\pi}$ is \'etale the natural morphism
\[ \widetilde{\pi}^{\ast} \BL_{P/E} \xrightarrow{\ \cong \ } \BL_K \]
is an isomorphism.

\begin{prop}
The composition
\begin{equation}
L \iota^{\ast} G^{\bullet}
\xrightarrow{L \iota^{\ast}(\varphi)}
\widetilde{\pi}^{\ast}(\BL_{P/E})
\xrightarrow{\ \cong \ } \BL_K
\label{newredthy} \end{equation}
is a symmetric perfect obstruction theory on $\CK$.
\end{prop}
\begin{proof}
By Proposition \ref{Proposition_Reduced_Perfect_Obstruction_Theory}
$\Ered^{\bullet} \to \BL_P$ is a perfect obstruction theory on $P$ and therefore
an isomorphism in $h^0$ and a surjection in $h^{-1}$.
By the long exact sequence in cohomology applied to \eqref{DIAG} we find
$\varphi : G^{\bullet} \to \pi^{\ast} \BL_{P/E}$ is an isomorphism in $h^0$ and a surjective in $h^{-1}$.
By a spectral sequence argument it follows
that
$L \iota^{\ast}(\varphi) : L \iota^{\ast} G^{\bullet} \to \widetilde{\pi}^{\ast}(\BL_{P/E})$
is an isomorphism in $h^0$ and a surjection in $h^{-1}$, and hence so is the composition \eqref{newredthy}.

By Lemma \ref{Lemma_Gbullet_perfect} $G^{\bullet}$ can be represented by
a two term complex $[G^{-1} \to G^0]$ of vector bundles $G^{-1}, G^0$. Hence
$L\iota^{\ast} G^{\bullet}$ is represented by $[ \iota^{\ast} G^{-1} \to \iota^{\ast} G^0 ]$
and therefore perfect of amplitude contained in $[-1,0]$. This shows \eqref{newredthy} is a perfect obstruction theory.
The symmetry of the obstruction theory now follows from Proposition \ref{Proposition_Gbullet_is_symmetric} by restriction.
\end{proof}

\subsection{The virtual class}
Let $[ \CK ]^{\text{vir}}$ be the virtual class on $\CK$ associated to the perfect obstruction theory \eqref{newredthy}, and let
\begin{equation} \label{virclassdef} [ P / E ]^{\text{vir}} \, = \, \frac{1}{|G|}\, \pi_{\ast} [\, \CK \, ]^{\text{vir}} \ \in A_{0}(P/E). \end{equation}

We give another expression for $[ P/E ]^{\text{vir}}$.
Consider the fiber diagram
\[
\begin{tikzcd}
X \times P \ar{d}{\rho} \ar{r}{\pi_P} & P \ar{d}{\pi} \\
(X \times P)/E \ar{r}{\pi_{P/E}} & P/E
\end{tikzcd}
\]
where we let $(X \times P)/E$ denote the quotient by the diagonal action.
There exist a complex $\overline{\II}$ on $(X \times P)/E$ such that
$\rho^{\ast}( \overline{\II} ) = \II$ is the universal stable pair on $X \times P$.
Define the complex
\[ \CH^{\bullet} = R \pi_{P/E \ast} \hom( \overline{\II}, \overline{\II} )_0[2]. \]
By an argument identical to \cite[Lemma 4.2]{HT} $\CH^{\bullet}$ is isomorphic to a $2$-term complex of locally free sheaves
$[ H^{-1} \to H^0 ]$ in degree $-1$ and $0$.
By flat base change we have
\begin{equation} \pi^{\ast} \CH^{\bullet} = \pi^{\ast} R \pi_{P/E \ast} \hom( \overline{\II}, \overline{\II} )_0[2]
 = R \pi_{P \ast} \hom(\II, \II)_0[2] = E^{\bullet}.
 \label{flat_base_change}
 \end{equation}

Let $h : Y \to P/E$ be any proper \'etale morphism of degree $\deg(h)$ from a scheme $Y$ which can be embedded into a smooth ambient scheme\footnote{
Since $\CK \to P/E$ is an \'etale map from a projective scheme, the set of such $h$ is non-empty.},
and let $c_F(Y)$ be the Fulton Chern class of $Y$, see Appendix \ref{Appendix_Fulton_Chern_Class}.
Consider the class
\begin{equation}
\frac{1}{\deg(h)} h_{\ast} \Big\{ s\big( (h^{\ast} \CH^{\bullet})^{\vee} \big) \cap c_F(Y) \Big\}_0
\, = \,
\frac{1}{\deg(h)} h_{\ast} \left\{ \frac{ c(h^{\ast} H_1) }{ c( h^{\ast} H_0 ) } \cap c_F(Y) \right\}_0
\label{hclass}
\end{equation}
in $A_0(P/E)$ 
where $H_1 = (H^{-1})^{\vee}$ and $H_0 = (H^0)^{\vee}$, and $\{ \ldots \}_0$ is the dimension $0$ component
of an element in the Chow ring of $P/E$.

\begin{prop} \label{Proposition_formula_for_vir_class} The class \eqref{hclass} is independent of $h$, and we have
\[ [ P / E ]^{\text{vir}} \, = \,
\frac{1}{\deg(h)} h_{\ast} \Big\{ s\big( (h^{\ast} \CH^{\bullet})^{\vee} \big) \cap c_F(Y) \Big\}_0 \,.
\]
\end{prop}

\begin{cor} The virtual class $[P/E]^{\text{vir}}$ is independent of the choice of isotrivial fibration $p: P \to E$. \end{cor}

\begin{proof} [Proof of Proposition \ref{Proposition_formula_for_vir_class}]
Let $Y \xrightarrow{h} P/E \xleftarrow{h'} Y'$ be proper \'etale morphism
from schemes which can be embedded into smooth ambient schemes. By taking the fiber product $Y \times_{P/E} Y'$
we can restrict to the case where there is an proper \'etale morphism $k : Y' \to Y$ such that $h' = h \circ k$.
We then have
\begin{align*}
& \frac{1}{\deg(h)} h_{\ast} \left\{ \frac{ c(h^{\ast} H_1) }{ c( h^{\ast} H_0 ) } \cap c_F(Y) \right\}_0 \\
= \ &
\frac{1}{\deg(h) \deg(k)} h_{\ast} k_{\ast} k^{\ast} \left\{ \frac{ c(h^{\ast} H_1) }{ c( h^{\ast} H_0 ) } \cap c_F(Y) \right\}_0 \\
= \ & 
\frac{1}{\deg(h')} h'_{\ast} \left\{ \frac{ c(h^{\prime \ast} H_1) }{ c( h^{\prime \ast} H_0 ) } \cap k^{\ast} c_F(Y) \right\}_0.
\end{align*}
But since $k$ is \'etale, we have $k^{\ast} c_F(Y) = c_F(Y')$ see Appendix \ref{Appendix_Fulton_Chern_Class}.
Hence the class \eqref{hclass} is independent of $h$.

Consider the \'etale morphism $\widetilde{\pi} : \CK \to P/E$ and let $\iota : K \to P$ be the inclusion.
Let $[G_0 \to G_1]$ be a presentation of $(G^{\bullet})^{\vee}$ on $P$ by locally free sheaves in degree $0$ and $1$.
Since $(G^{\bullet})^{\vee}$ differs from $(E^{\bullet})^{\vee}$ only by trivial bundles $\CO_P$,  
we have $s((G^{\bullet})^{\vee}) = s( (E^{\bullet})^{\vee} )$
and therefore
\begin{equation} \label{12345}
s\big( (L \iota^{\ast} G^{\bullet})^{\vee} \big)
= s\big( (L \iota^{\ast} E^{\bullet})^{\vee} )
 = s\big( (\widetilde{\pi}^{\ast} \CH^{\bullet})^{\vee} \big)
\end{equation}
where we used \eqref{flat_base_change} in the last step.
By Siebert's virtual class formula \cite{S} (compare also \cite[Appendix C]{PT2})
we have
\begin{equation}
[ \CK ]^{\text{vir}}
=
\Big\{ s\big( (L \iota^{\ast} G^{\bullet})^{\vee} \big) \cap c_F(\CK) \Big\}_0.
\label{Siebertsformula} \end{equation}
The claim now follows from \eqref{12345}, \eqref{Siebertsformula} and the definition \eqref{virclassdef} of $[P/E]^{\text{vir}}$.
\end{proof}

\subsection{Proof of Theorem \ref{Theorem_Comparision_to_reduced_class}}
Let $m : E \times \CK \to P, (x, I) \mapsto I+x$ be the restriction of the translation
action to $E \times \CK$.
Using the existence of the inclusion $\iota : \CK \to P$ one proves the diagram
\[
\begin{tikzcd}
E \times \CK \ar{r}{\pi_K} \ar{d}{m} & \CK \ar{d}{\widetilde{\pi}} \\
P \ar{r}{\pi} & P/E,
\end{tikzcd}
\]
where $\pi_K$ is the projection to the second factor, is a fiber diagram.

Since the reduced obstruction theory $\Ered^{\bullet}$ differs from $E^{\bullet}$
only by a trivial bundle $\CO_P$ we have
\[ [ P ]^{\text{red}} \, = \, \Big\{ s\big( (E^{\bullet})^{\vee} \big) \cap c_F(P) \Big\}_1. \]
Because $m$ is an \'etale map of degree $\deg(\widetilde{\pi}) = |G|$ we have therefore
\[ [ P ]^{\text{red}} = \frac{1}{|G|} m_{\ast} \Big\{ s\big( (m^{\ast} E^{\bullet})^{\vee} \big) \cap m^{\ast} c_F(P) \Big\}_1. \]
By invariance of the Fulton Chern class under \'etale pullback we have
$m^{\ast} c_F(P) = c_F(\CK \times E) = \pi_{\CK}^{\ast} c_F(\CK)$, and by \eqref{flat_base_change} we have
$m^{\ast} E^{\bullet} = m^{\ast} \pi^{\ast} \CH^{\bullet} = \pi_K^{\ast} \widetilde{\pi}^{\ast} \CH^{\bullet}$.
Therefore
\begin{align*}
[ P ]^{\text{red}}
& = \frac{1}{|G|} m_{\ast} \Big\{ s\big( (\pi_K^{\ast} \widetilde{\pi}^{\ast} \CH^{\bullet})^{\vee} \big) \cap \pi_K^{\ast} c_F(\CK) \Big\}_1 \\
& = \frac{1}{|G|} m_{\ast} \pi_K^{\ast} \Big\{ s\big( (\widetilde{\pi}^{\ast} \CH^{\bullet})^{\vee} \big) \cap c_F(\CK) \Big\}_0 \\
& = \frac{1}{|G|} \pi^{\ast} \widetilde{\pi}_{\ast} \Big\{ s\big( (\widetilde{\pi}^{\ast} \CH^{\bullet})^{\vee} \big) \cap c_F(\CK) \Big\}_0 \\
& = \pi^{\ast} [ P/E ]^{\text{vir}}
\end{align*}
where we applied Proposition \ref{Proposition_formula_for_vir_class} in the last step. The proof is complete. \qed

\subsection{Proof of Theorem \ref{4fijdsdf}}
Let $S \xleftarrow{\pi_1} S \times E \xrightarrow{\pi_2} E$ be the projections and recall the fiber diagram
\[
\begin{tikzcd}
X & \ar{l}[swap]{\pi_X} X \times P \ar{r}{\pi_{P}} \ar{d}{\rho} & P \ar{d}{q} \\
& (X \times P)/E \ar{r}{\pi_{P/E}} & P / E.
\end{tikzcd}
\]
Let $D \in H^2(S,\BQ)$ be a class satisfying $\langle D, \beta \rangle = 1$,
and let $\omega \in H^2(E)$ be the class of a point.
Let $\II$ be the universal stable pair on $X \times P$.
By definition of $\tau_0( \cdot )$ the invariant $\widetilde{\mathsf{N}}^{X}_{n, (\beta,d)}$
is the degree of
\begin{equation}
\mathbf{Z}
=
\Big(
-\mathrm{ch}_2( \II ) \cdot
\pi_X^{\ast}(\pi_{1}^{\ast}(D) \cup \pi_{2}^{\ast}(\omega))
\Big)
\cap \pi_P^{\ast} [ P ]^{\text{red}}. \label{PxX_cycle} \end{equation}
Hence it is enough to show $(\pi_{P/E} \circ \rho)_{\ast} \mathbf{Z} = [ P/E ]^{\text{vir}}$.

By Theorem \ref{Theorem_Comparision_to_reduced_class} we have
\[ \pi_{P}^{\ast} [ P ]^{\text{red}} = \rho^{\ast} \pi_{P/E}^{\ast} [ P/E ]^{\text{vir}}. \]
Let $\overline{\II}$ be the complex on $(X \times P)/E$ with $\rho^{\ast} \overline{\II} = \II$.
Then 
\[ \ch_2( \II ) = \rho^{\ast} \ch_2(\overline{\II}) \,. \]
The composition $\pi_2 \circ \pi_X$ descends to a map 
$\pi_S : (P \times X)/E \to S$. Hence
\[ \pi_X^{\ast}( \pi_1^{\ast}(D)) = \rho^{\ast}( \pi_{S}^{\ast}( D ) ) \,. \]
Applying the push-pull formula, we therefore have
\[ \rho_{\ast} \mathbf{Z} = \rho_{\ast}( \pi_E^{\ast}(\omega) \cap \rho^{\ast}( \alpha ) ) \]
where $\pi_E = \pi_2 \circ \pi_X$ and
\[ \alpha = \Big( - \mathrm{ch}_2( \overline{\II} ) \cup \pi_{S}^{\ast}(D) \Big) \cap \pi_{P/E}^{\ast} [ P/E ]^{\text{vir}} \,. \]

We will identify $(X \times P)/E$ with the product $S \times P$ via the isomorphism
\begin{equation}
(X \times P)/E \to P \times S, \ [ ((s,e), I) \mapsto (s, I-e) ]. \label{identification}
\end{equation}
Under \eqref{identification} the map $\rho$ is identified with
\[ \id_S \times m : S \times E \times P \to S \times P, \ (s,e, I) \mapsto (s, I-e) \,. \]
Let $\iota : (S \times 0_E) \times P \hookrightarrow X \times P$ be the inclusion. We conclude
\[
\rho_{\ast} \mathbf{Z}
= \rho_{\ast}( \pi_E^{\ast}(\omega) \cap \rho^{\ast}( \alpha ) )
= \rho_{\ast}( \iota_{\ast} \iota^{\ast} \rho^{\ast}( \alpha ) )
= (\rho \circ \iota)_{\ast} (\rho \circ \iota)^{\ast} (\alpha)
= \alpha \]
where the last step holds since $\rho \circ \iota$ is the identity.

Because $\pi_{P/E}$ has relative dimension $3$ with fiber $X$,
and $\ch_2(I) = -\beta$ for every stable pair $I \in P$,
we conclude
\[
\pushQED{\qed}
(\pi_{P/E} \circ \rho)_{\ast} \mathbf{Z}
= \pi_{P/E\, \ast} \alpha
= \langle D, \beta \rangle [ P/E ]^{\text{vir}} = [ P/E ]^{\text{vir}} \,. \qedhere
\popQED
\]

\section{Applications} \label{Section_Application}
\subsection{Overview}
Let $X$ be a Calabi-Yau threefold, that is, a smooth projective threefold satisfying
\[ \bigwedge^3 \Omega_X \cong \CO_X. \]
In particular, we allow $H^1(X, \CO_X)$ to be non-zero.
We also
fix a Cohen-Macaulay curve
\begin{equation} C \subset X \label{CMcurve} \end{equation}
with homology class $\beta \in H_2(X,\BZ)$.

In Section~\ref{Section_C_local_invariants} we define $C$-local Donaldson-Thomas and stable pair invariants of $X$,
which can be thought of as the contributions of the curve $C$ to the DT/Pairs invariants of $X$ in class $\beta$.
Our main theorem is a DT/PT and rationality result for $C$-local invariants.
Theorem \ref{Theorem_applications} then follows by
integrating over the Chow variety of curves up to translation.

Throughout we follow closely the articles \cite{Br1, Br2} by Bridgeland,
but also refer to \cite{Joyce1, Joyce2,Joyce-Song,T16} for further details.

\subsection{The Chow variety} \label{Subsection:Chow variety}
Let
\[ \Chow(X) \]
be the Chow variety parametrizing $1$-dimensional cycles on $X$
as constructed by Koll\'ar in \cite[I.1.3]{K}.
Let $F$ be a coherent sheaf on $X$ supported in dimension $\leq 1$.
The fundamental $1$-cycle of $F$ is
\begin{equation*}
[F] = \sum_{\eta} \left( \mathrm{length}_{\CO_{X,\eta}}(F) \right) \overline{ \{ \eta \} } \  \in \Chow(X), \label{CYCLEDEF}
\end{equation*}
where $\eta$ runs over the set of all codimension $2$ points of $X$.
If $Z \subset X$ is a curve, we set $[Z] := [ \CO_Z ]$.

The Chow variety has the following property.\footnote{I am grateful to Andrea Ricolfi for pointing out
an error in the statement of the Lemma in an previous version.}

\begin{lemma} \label{Chow-Lemma}
Let $S$ be a scheme locally of finite type, and let $\CF$ be a coherent sheaf on $S \times X$ which is flat over $S$
and whose restriction $\CF_s$ to the fiber over every $s \in S$ is supported in dimension $\leq 1$.
Let $\nu : S_{\mathrm{sn}} \to S$ be the semi-normalization of $S$.
Then there exist a Hilbert-Chow morphism
\[ \rho: S_{\mathrm{sn}} \to \Chow(X) \]
with $\rho( s ) = [ \CF_s ]$ for all closed points $s \in S$.
\end{lemma}
\begin{proof}
See \cite[I.6.3]{K} for the case of the Hilbert scheme.
We also refer to \cite{R} for a further discussion of Chow varieties.
\end{proof}

The seminormalization $\nu : S_{\mathrm{ns}} \to S$
is an isomorphism of underlying topological spaces.
Hence the lemma yields a continous map
$|\rho| : |S| \to |\Chow(X)|$ from the underlying topological spaces of $S$ to the underlying topological space of $\Chow(X)$.
In our application belows, we will be concerned with taking the fiber of a closed subset under $|\rho|$
and endow it with the induced reduced subscheme structure
(equivalently, we could take the preimage under $\rho$ with the reduced scheme structure,
and than the scheme theoretic image thereof under the proper map $\nu$).
For convenience we will often drop $| - |$ from notation and call $|\rho|$ the topological Hilbert-Chow morphism.

\subsection{$C$-local invariants} \label{Section_C_local_invariants}
Let $P_n(X, \beta)$ be the stable pairs moduli space.
The fixed Cohen-Macaulay curve \eqref{CMcurve} determines a closed point
\[ [C] \in \Chow(X). \]
The subscheme of stable pairs supported on $C$ is the fiber\footnote{
Here $\rho_P$ is the Hilbert-Chow morphism
viewed as a map from the underlying topological space of $P_n(X,\beta)$ to that of $\Chow(X)$,
as discussed in Section~\ref{Subsection:Chow variety}.}
\[ P_n(X,C) = \rho_P^{-1}( [C] ) \subset P_n(X, \beta) \]
endowed with its reduced subscheme structure.
Let $\nu : P_n(X, \beta) \to \BZ$ be the Behrend function
on the ambient space\footnote{The restriction of $\nu$ to $P_n(X,C)$ might differ from the Behrend function of $P_n(X, C)$.}.
We define $C$-local stable pair invariant of $X$ by
\[ \mathsf{P}_{n,C} = \int_{P_n(X, C)} \nu \dd{e} . \]
The integral $\mathsf{P}_{n,C}$ may be thought of as the contribution of the curve $C$
to the stable pair invariant of $X$ in class $\beta$.

Similarly, let 
$\Hilb_X(C,n) = \rho_H^{-1}( [C] )$
be the fiber of the Hilbert-Chow map $\rho_H$ over $[C]$
endowed with the reduced subscheme structure.
The $C$-local Donaldson Thomas invariant of $X$ is defined by
\[ \mathsf{DT}_{n,C} = \int_{\Hilb_X(C,n)} \nu \dd{e} \]
where $\nu : \Hilb_X(\beta,n) \to \BZ$ is the Behrend function on the ambient space.

Define generating series
\[
\mathsf{DT}_{C}(q) = \sum_n \mathsf{DT}_{n,C} q^n
\quad \text{ and } \quad
\mathsf{PT}_{C}(q) = \sum_n  \mathsf{P}_{n,C} q^n 
\]
of $C$-local Donaldson-Thomas and stable pair invariants of $X$ respectively.
In particular, for $C$ the empty curve, we recover the generating series of degree $0$ Donaldson-Thomas invariants,
\[ \mathsf{DT}_{0}(q) := \mathsf{DT}_{\varnothing}(q) = \int_{ \Hilb_X(n) } \nu \dd{e} = M(-q)^{\chi(X)}, \]
where $\Hilb_X(n)$ is the Hilbert scheme of $0$-dimensional length $n$ subschemes in $X$
and $M(q) = \prod_{k \geq 1} (1-q^k)^{-k}$ is the MacMahon function.

\begin{thm} \label{Theorem_Clocal} \hfill
\begin{enumerate}
 \item[(i)] The $C$-local DT/PT correspondence holds:
 \[ \mathsf{DT}_{C}(q) = \mathsf{PT}_{C}(q) \cdot \mathsf{DT}_{0}(q) \]
 \item[(ii)] The series $\mathsf{PT}_{C}(q)$ is the Laurent expansion of a rational function in $q$,
 which is invariant under the transformation $q \mapsto q^{-1}$.
\end{enumerate}
\end{thm}

Theorem~\ref{Theorem_Clocal} is a localization
of results of Bridgeland \cite{Br1} to the fixed curve $C \subset X$.
While the statements of Theorem \ref{Theorem_Clocal} appear to be known,
we have found no direct reference and we will give a proof below.

Related cases have appeared elsewhere.
The analog of part (i) in the Euler characteristic case (with no Behrend function weighting)
can be found in \cite{ST}.
The rationality for $C$-local invariants in case $\beta$ irreducible
is proven in \cite[Section 3]{PTBPS}.
The $C$-local case for parabolic stable pair invariants is discussed in \cite[Section 4.4]{Tpar} \cite{Tpar2}.

\subsection{Stack of coherent sheaves}
Let $\Coh(X)$ be the category of coherent sheaves on $X$ and let
\[ \Coh_{\leq 1}(X) \subset \Coh(X) \]
be the full subcategory of sheaves supported in dimension $\leq 1$.

Let $C_1, \dots, C_N$ be the (reduced) irreducible components of $C$.
We say $F \in \Coh_{\leq 1}(X)$ is supported on $C$ in dimension $1$ if
\[ [ F ] = \sum_{i} m_i [C_i] \]
for some integers $m_1, \ldots, m_N \geq 0$.
Let
\[ \Coh_C(X) \subset \Coh_{\leq 1}(X) \]
be the full subcategory of sheaves supported on $C$ in dimension $1$.
For any exact seqence
\[ 0 \to E_1 \to E_2 \to E_3 \to 0 \]
in $\Coh_{\leq 1}(X)$ we have $[E_2] = [E_1] + [E_3]$. Hence,
$E_2$ is in $\Coh_C(X)$ if and only if $E_1$ and $E_3$ is in $\Coh_C(X)$.
In other words $\Coh_C(X)$ is closed under extensions and taking kernel and cokernel.

Let $\CA$ be the stack of coherent sheaves on $X$ supported in dimension $\leq 1$.
The stack $\CA$ is algebraic and locally of finite type over $\BC$. 
Let $| \CA |$ be the set of points in $\CA$ endowed with the Zariski topology.

\begin{lemma} The subset
\[ Z_C \subset | \CA | \]
of points corresponding to sheaves supported on $C$ in dimension $1$ is closed.
\end{lemma}
\begin{proof}
Let $\pi : U \to \CA$ be an atlas for $\CA$, where $U$ is a scheme locally of finite type.
Let $\rho : U \to \Chow(X)$ be the (topological) Hilbert-Chow map obtained by Lemma~\ref{Chow-Lemma}
from the family of coherent sheaves on $U \times X$ corresponding to $\pi$.
Let $\widetilde{Z}_C$ be the preimage under $\rho$ of the discrete (and closed) subset
\[ \bigsqcup_{m_1, \dots, m_N \geq 0} \Big\{ \sum_i m_i [C_i] \Big\} \subset \Chow(X)\]
endowed with the reduced subscheme structure.
Since $\pi$ is open, surjective and continuous and $\widetilde{Z}_C = \pi^{-1}(Z_C)$ is closed,
also $Z_C = \pi( \widetilde{Z}_C )$ is closed.
\end{proof}

Define the stack of coherent sheaves supported on $C$ as the closed reduced substack
\[ \iota : \CA_C \hookrightarrow \CA \]
with underlying topological space $Z_C$.
Since $\iota$ is a closed immersion, $\iota$ is of finite type.
Hence, $\CA_C$ is an algebraic stack locally of finite type.

%
%
\subsection{Hall algebras}
Let $K(\mathrm{St}/\CF)$ denote the Grothendieck ring of stacks over an algebraic stack $\CF$
\cite{Br2}.
The Hall algebra $H( \CA )$ is the $K(\mathrm{St}/\Spec \BC)$-module
$K(\mathrm{St}/\CA)$ endowed with
the Hall algebra product $\ast$ defined as follows.
Let $\Ex$ be the stack of short exact sequences in $\CA$,
and consider the diagram
\[
\begin{tikzcd}
\Ex \ar{r}{b} \ar{d}{(a,c)} & \CA \\
\CA \times \CA
\end{tikzcd}
\]
where $a,b,c$ sends families of exact sequences
$0 \to A \to B \to C \to 0$
to $A,B,C$ respectively.
Then we set
\[ [X_1 \to \CA] \ast [X_2 \to \CA ] = \big[ (X_1 \times X_2) \times_{\CA \times \CA, (a,c)} \Ex \to \Ex \xrightarrow{b} \CA \big] \]
for all $[X_i \to \CA] \in K(\mathrm{St}/\CA)$ where $X_i$ are algebraic stacks of finite type with affine geometric stabilizers.

Let $\Ex_C$ be the stack of short exact sequences in $\CA_C$, and consider
\[
\begin{tikzcd}
\Ex_C \ar{r}{b'} \ar{d}{(a',c')} & \CA_C \\
\CA_C \times \CA_C
\end{tikzcd}
\]
where $a',b',c'$ sends families of exact sequences
$0 \to A \to B \to C \to 0$
to $A,B,C$ respectively.
To define a Hall algebra $H(\CA_C)$ with good properties we require the following Lemma,
compare \cite[Assumption 7.1/8.1]{Joyce1}.

\begin{lemma} \label{123} \label{Lemma_algfintype} The stacks $\CA_C$ and $\Ex_C$ are algebraic stacks locally of finite type, and the morphisms
\[ b' : \Ex_C \to \CA_C \quad \text{ and } \quad (a',c') : \Ex_C \to \CA_C \]
are of finite type
\end{lemma}
\begin{proof}
We have already seen that $\CA_C$ is algebraic and locally of finite type.
We consider $\Ex_C$. Since $(a,b,c)$ is a iso-fibration \cite[Appendix]{Br2}, we have the fiber square
\[
\begin{tikzcd}
\Ex_C \ar{d}{(a',b',c')} \ar{r} & \Ex \ar{d}{(a,b,c)} \\
(\CA_C)^3 \ar{r} & ( \CA )^3.
\end{tikzcd}
\]
Hence $\Ex_C$ is the fiber product of algebraic stacks and therefore algebraic.
Moreover, since $(a,b,c)$ is of finite type, so is $(a',b',c')$, and hence $\Ex_C$ is locally of finite type.

We show $b'$ is of finite type. The case $(a', c')$ is similar. Consider the fiber diagram
\[
 \begin{tikzcd}
 \widetilde{\Ex} \ar{d}{\widetilde{b}} \ar{r}{j} & \Ex \ar{d}{b} \\
 \CA_C \ar{r} & \CA.
\end{tikzcd}
\]
Since $b$ is of finite type \cite{Br1},\cite[9.4]{Joyce1} so is $\widetilde{b}$.
The stack $\Ex_C$ is the fiber product
\[
\begin{tikzcd}
\Ex_C \ar{r}{\iota'} \ar{d} & \widetilde{\Ex} \ar{d}{( a\circ j,\, b\circ j)} \\
(\CA_C)^2 \ar{r}{(\iota, \iota)} & \CA^2.
\end{tikzcd}
\]
Since $(\iota, \iota)$ is of finite type (as a closed immersion) so is $\iota'$.
We conclude the composition
$b' = \widetilde{b} \circ \iota' : \Ex_C \to \CA_C$
is of finite type.
\end{proof}

Define the Hall algebra $H(\CA_C)$ of $\CA_C$
as the group
$K(\mathrm{St}/\CA_C)$
together with the Hall algebra product
\[ [X_1 \to \CA_C ] \ast [X_2 \to \CA_C ] = \big[ (X_1 \times X_2) \times_{\CA_C \times \CA_C, (a',c')} \Ex_C \to \Ex_C \xrightarrow{b'} \CA_C \big]. \]
By Lemma \ref{123} the product $\ast$ is well-defined \cite{Br2}.

Consider the closed immersion 
$\iota : \CA_C \to \CA$.
Since $\iota$ is of finite type
we have maps of $K(\mathrm{St}/\BC)$-modules \cite[3.5]{Br2}
\[ \iota^{\ast} : H(\CA) \to H(\CA_C), \quad \quad \iota_{\ast} : H(\CA_C) \to H(\CA). \]

\begin{lemma} \label{Lemma_iota_algebra_morphism} The maps $\iota^{\ast}$ and $\iota_{\ast}$ are algebra homomorphisms. \end{lemma}
\begin{proof}
Let $i=1,2$ let $x_i = [X_i \to \CA] \in H(\CA)$
with $X_i$ an algebraic stack of finite type with affine stablizers.
We will prove
\[ \iota^{\ast}(x_1 \ast x_2) = \iota^{\ast}(x_1) \ast \iota^{\ast}(x_2). \]
The case $\iota_{\ast}$ is similar.

The product $x_1 \ast x_2$ is the element $w = [W \to \CA]$ defined by the top row
of the fiber diagram
\[
\begin{tikzcd}
W \ar{r} \ar{d} & \Ex \ar{r}{b} \ar{d}{(a,c)} & \CA \\
X_1 \times X_2 \ar{r} & \CA \times \CA.
\end{tikzcd}
\]
Pulling back the top row via $\iota : \CA_C \to \CA$
defines the element 
\begin{equation} [ \iota^{\ast}(w)] = [ \widetilde{W} \to \widetilde{\Ex} \to \CA_C ], \label{13134} \end{equation}
where $\widetilde{\Ex}$ is as in the proof of Lemma~\ref{Lemma_algfintype}.

With $\iota^{\ast}(X_i) = X_i \times_{\CA} \times \CA_C$ for $i=1,2$
the top row of
\[
\begin{tikzcd}
W_C \ar{d} \ar{r} & \Ex_C \ar{r}{b'} \ar{d}{(a',c')} & \CA_C \\
\iota^{\ast}(X_1) \times \iota^{\ast}(X_2) \ar{r} & \CA_C \times \CA_C
\end{tikzcd}
\]
defines the element
$\iota^{\ast}(x_1) \ast \iota^{\ast}(x_2) = [ W_C \to \Ex_C \to \CA_C ]$.

By a diagram chase there is a morphism $f: W_C \to W$
commuting with the maps to $\CA_C$.
Since $\Coh_{C}$ is closed under taking extension and taking kernel and cokernel,
$f$ is a geometric bijection. Hence
$[ W_C \to \CA_C ] = [ \widetilde{W} \to \CA_C ]$
and the claim follows.
\end{proof}

\subsection{Integration maps and removing the $H^1(X,\CO_X) = 0$ requirement}
Consider the commutative ring
\[ \Lambda = K(\mathrm{Var})\left[\BL^{-1} , ( \BL^n + \ldots + 1 )^{-1}, n \geq 1 \right]. \]
Define the subalgebra of regular elements
\[ H^{\text{reg}}(\CA) \subset H(\CA) \]
as the $\Lambda$-module generated by all elements $[X \to \CA]$ where $X$ is a variety.
The semi-classical limit of $H^{\text{reg}}(\CA)$ is
\[ H^{\mathrm{sc}}(\CA) = H^{\text{reg}}(\CA)/(\BL - 1) H^{\text{reg}}(\CA) \]
endowed with the usual Poisson bracket \cite{Br2}.
Let also
\[ \Delta = H_2(X,\BZ) \oplus H_0(X, \BZ) = H_2(X,\BZ) \oplus \BZ \]
and let 
$\BC[\Delta]$ be the monoid algebra of $\Delta$
with basis $q^n v^{\beta}$ for all $n \in \BZ$ and $\beta \in H_2(X,\BZ)$,
endowed with the trivial Poisson structure.

We define an integration map
\begin{equation} \Upsilon : H^{\mathrm{sc}}(\CA) \to \BC[\Delta]. \label{YYY} \end{equation}
The stack $\CA$ splits into a disjoint union
\[ \CA = \bigsqcup_{(\beta,n) \in \Delta} \CA_{\beta,n} \]
where $\CA_{\beta,n}$ parametrizes sheaves of Chern character $(0,0, \beta,n)$.
Then, given a variety $V$ and a map $V \to \CA_{\beta,n}$ we define
\[ \Upsilon\big( [V \xrightarrow{f} \CA_{n,\beta}] \big) \, = \, e\big( V, f^{\ast}(\nu_{\CA}) \big) \cdot q^n v^{\beta} \]
where $\nu_{\CA} : \CA \to \BC$ is the Behrend function on $\CA$
and $e(\, \cdot \, )$ is the topological Euler characteristic.

In \cite[Section 5]{Br2} Bridgeland (based on the work of Joyce-Song \cite{Joyce-Song})
proves that $\Upsilon$ is well-defined and, in case $H^1(X,\CO_X) = 0$, is a \emph{Poisson algebra homomorphism}.
We explain how the last result can be extended to the case of non-vanishing $H^1(X,\CO_X)$.
The key step in \cite[Section 5]{Br2} is to prove the Behrend function identitites \cite[Eqns. (13), (14)]{Br2}.
These are proven in \cite[Ch. 10]{Joyce-Song}
and the proof relied upon the $H^1(X,\CO_X)$ condition only through \cite[Thm 5.5]{Joyce-Song}
which states that the stack of coherent sheaves $\Coh(X)$ on $X$
at any point $[E]$ can be written locally as a $G$-invariant critical locus
(with $G$ the maximal reductive subgroup of the automorphism group of $[E]$).
Hence, it is enough to show \cite[Thm 5.5]{Joyce-Song} also for $H^1(X,\CO_X) \neq 0$.

The original proof of \cite[Thm 5.5]{Joyce-Song} by Joyce-Song used gauge-theoretic arguments
and required $H^1(X,\CO_X) = 0$ crucially, see \cite[Rmk 5.2]{Joyce-Song}.
A second (and more general) proof was recently established in \cite{T16}.
The new input here is the existence of a $(-1)$-shifted symplectic structure
on a derived extension of $\Coh(X)$ shown in \cite{PTVV}.
Combining this with an algebraic Darboux Theorem \cite{BBBBJ}
leads to the desired local $G$-invariant charts for $\Coh(X)$.
Since \cite{PTVV} and the following discussion in \cite{T16} do not require $H^1(X, \CO_X) = 0$,
the second proof works also $H^1(X, \CO_X) \neq 0$.
Hence the Behrend function identities hold
and \eqref{YYY} is a Poisson algebra morphism also in the case $H^1(X, \CO_X) \neq 0$.

\subsection{The integration map for $\CA_C$}
Let $H^{\text{reg}}(\CA_C) \subset H(\CA_C)$ be the regular subalgebra,
that is the $\Lambda$-module generated by $[V \to \CA_C]$ with $V$ a variety,
and let $H^{\mathrm{sc}}(\CA_C)= H^{\text{reg}}(\CA_C)/(\BL - 1)H^{\text{reg}}(\CA_C)$
be the semiclassical limit.

Let $C_1, \dots, C_N$ be the irreducible components of $C$ and let
\[ \Delta_C = H_2(C,\BZ) \oplus \BZ = \oplus_{i=1}^{N} \BZ [C_i] \oplus \BZ, \]
where we identify an element $\gamma  \in H_2(C,\BZ)$ with the corresponding $1$-cycle on $X$.
We have the disjoint union
\[ \CA_C = \bigsqcup_{(\gamma,n) \in \Delta_C} \CA_{C,\gamma,n} \]
with $\CA_{C,\gamma,n}$ parametrizing sheaves supported on $C$ of
cycle class $\gamma$ and Euler characteristic $n$.
Define the integration map
\[ \Upsilon_C : H^{\mathrm{sc}}(\CA_C) \to \BC[\Delta_C] \]
by 
\[ \Upsilon_C([V \xrightarrow{f} \CA_C]) = e(V, (\iota \circ f)^{\ast} \nu_{\CA}) \cdot q^n v^{\gamma} \]
for all maps $V \to \CA_{C,\gamma,n}$ with $V$ a variety.
Note the Euler characteristic is weighted by the Behrend function of $\CA$.

Alternatively,
consider $H(\CA)$ (resp. $H(\CA_C)$) as $\Delta$-graded (resp. $\Delta_C$-graded) algebras with respect to the Chern character,
and likewise for $H^{\text{reg}}$ and $H^{\text{sc}}$.
Then
$\Upsilon_C : H^{\mathrm{sc}}(\CA_C) \to \BC[\Delta_C]$
is the unique $\Delta_C$-\emph{graded} homorphism such that
\[ j_{\ast} \circ \Upsilon_C = \Upsilon \circ \iota_{\ast} \]
where $j : C \to X$ is the inclusion and $j_{\ast} : \BC[\Delta_C] \to \BC[\Delta]$ is the map sending $v^{[C_i]}$ to $v^{j_{\ast}[C_i]}$.
Hence, by Lemma \ref{Lemma_iota_algebra_morphism}, $\Upsilon_C$ is a morphism of Poisson algebras.


\subsection{The $C$-local DT/PT correspondence}
We prove Theorem \ref{Theorem_Clocal} (i). As in \cite{Br1, Br2} we will work with a completion of the Hall algebra $H(\CA)$
with respect to the set of \emph{Laurent} subsets of $H_2(X,\BZ) \oplus \BZ$,
and similarly for $H(\CA_C)$ and $H^{\text{reg}}$.
All formulas below shall be understood in this larger algebra.

Following the notation of \cite{Br2} we let
\[
\CH_{\leq 1} = \sum_{\beta,n} [\Hilb_X(\beta,n) \to \CA], \quad
\CH_{\leq 1}^{\sharp} = \sum_{\beta,n} [P_n(X,\beta) \to \CA], \]
\[ \CH_0 = \sum_n [ \Hilb_X(0,n) \to \CA ],
\]
where the maps in the Hilbert scheme case are defined by sending a subscheme to its structure sheaf,
and in the pairs case we send $[\CO_X \to F]$ to $F \in \CA$.
The elements $\CH_{\leq 1}, \CH_{\leq 1}^{\sharp}, \CH_0$
are regular elements (in some completion) of the Hall algebra $H(\CA)$.
We also let $\CP \subset \CA$ be the open substack of sheaves of dimension $0$,
and write $1_{\CF} = [\CF \to \CA]$ for an open substack $\CF \subset \CA$.

\begin{proof}[Proof of Theorem \ref{Theorem_Clocal} (i)]
By \cite[Prop. 6.5]{Br2} we have the identity
\begin{equation*} \CH_{\leq 1} \ast 1_{\CP} = \CH_0 \ast 1_{\CP} \ast \CH_{\leq 1}^{\sharp} \label{564564} \end{equation*}
Applying $\iota^{\ast}$ yields by Lemma \ref{Lemma_iota_algebra_morphism}
\begin{equation} \label{MA} \iota^{\ast}(\CH_{\leq 1}) \ast \iota^{\ast}(1_{\CP}) = \iota^{\ast}(\CH_0) \ast \iota^{\ast}(1_{\CP}) \ast \iota^{\ast}(\CH_{\leq 1}^{\sharp}). \end{equation}

Recall the grading on $H(\CA_C)$ by Chern character.
We have the $(\ast, [C])$-components
\begin{align*}
\CH_C & := \Big[ \iota^{\ast}(\CH_{\leq 1}) \Big]_{[C]} = \sum_{n} \big[ \Hilb_X(C,n) \to \CA_C \big] \\
\CH^{\sharp}_{C} & := \Big[ \iota^{\ast}(\CH_{\leq 1}^{\sharp}) \Big]_{[C]} = \sum_n \big[ P_n(X, C) \to \CA_C \big].
\end{align*}
The $(\ast, C)$-component of \eqref{MA} hence reads
\begin{equation} \CH_C \ast \iota^{\ast}(1_{\CP}) = \iota^{\ast}(\CH_0) \ast \iota^{\ast}(1_{\CP}) \ast \CH^{\sharp}_{C}. \label{1234522} \end{equation}
We apply $\iota_{\ast}$ and use Lemma \ref{Lemma_iota_algebra_morphism} again. Since
$\Coh_C(X)$ contains all sheaves supported in dimension~$0$ the map
$\CP \times_{\CA} \CA_C \to \CP$ is a geometric bijection, and similarly for $\CH_0$. Hence
\[ \iota_{\ast} \iota^{\ast}( 1_{\CP} ) = 1_{\CP}, \quad \quad \iota_{\ast} \iota^{\ast}( \CH_0 ) = \CH_0. \]
Inserting into \eqref{1234522} yields
the equality
\begin{equation} \iota_{\ast} \CH_C = \CH_0 \ast 1_{\CP} \ast \iota_{\ast}(\CH^{\sharp}_{C}) \ast 1_{\CP}^{-1} \label{MDFSDGSFG} \end{equation}
By \cite[Thm 3.1]{Br1} we have
\[
\Upsilon( \iota_{\ast} \CH_C ) = \mathsf{DT}_{C}(q), \quad \quad \Upsilon( \iota_{\ast} \CH^{\sharp}_{C} ) = \mathsf{PT}_{C}(q) 
\]
hence the claim follows from \eqref{MDFSDGSFG} by applying the integration map $\Upsilon$.
\end{proof}

\subsection{$C$-local generalized DT invariants}
Let $H$ be an ample divisor on $X$.
For all $\alpha = (\beta,n) \in \Delta$ define the slope
\[ \mu(\alpha) = \frac{ \ch_3(\alpha) }{ \ch_2(\alpha) \cdot H } = \frac{n}{\beta \cdot H}. \]
We have the induced slope function $\mu_C$ defined for all $\alpha = (\gamma,n) \in \Delta_C$
by
\[ \mu_C(\alpha) = \mu( j_{\ast}(\alpha)) = \frac{n}{\gamma \cdot \iota^{\ast}(H)} \]
where $j_{\ast} : \Delta_C \to \Delta$ is the pushforward map defined by $j:C \hookrightarrow X$.
Since $\mu$ defines a weak stability condition on $\CA$,
so does $\mu_C$ on $\CA_C$.

For all $\alpha \in \Delta_C$ let
\[ \CM^{\text{ss}}_C(\alpha) \subset \CA_{C,\alpha} \]
be the open substack of $\mu_C$-semistable sheaves of class $\alpha$.
Since $\CM^{\text{ss}}_C(\alpha)$ is a closed substack of the stack
of semistable sheaves in $\CA$ of class $j_{\ast} \alpha$ and the latter is of finite type,
also $\CM^{\text{ss}}_C(\alpha)$ is of finite type.

Let $\delta_{\alpha} = [ \CM^{\text{ss}}_C(\alpha) \to \CA_C ]$, and define elements
$\epsilon_{\alpha}$ by setting
\[
\sum_{\substack{\alpha \in \Delta_C \\ \mu_C(\alpha) = \mu}} \epsilon_{\alpha}
=
\log\Big( 1 + \sum_{\substack{\alpha \in \Delta_C \\ \mu_C(\alpha) = \mu}} \delta_{\alpha} \Big)
\]
for all $\mu \in \BR \cup \{ \infty \}$. By Joyce's theory of virtual indecomposables, see \cite[3.2]{Joyce-Song} or \cite{Joyce3},
the element
$\eta_{\alpha} = [ \BC^{\ast} ] \epsilon_{\alpha}$ lies in $H^{\text{reg}}(\CA)$.
The $C$-local generalized Donaldson-Thomas invariants $\mathsf{N}_{\gamma,n}$ of $X$ in class $(\gamma,n) \in \Delta_C$ are defined by
\begin{equation}
\Upsilon_C( \eta_{(\gamma,n)} ) = - \mathsf{N}_{\gamma,n} v^{\gamma} q^n.
\label{35232} \end{equation}

\begin{lemma}[\cite{Joyce-Song}, \cite{Br1}, \cite{Tpar2}] \label{Lemma_Relations_NN}
Let $\gamma = \sum_i m_i [C_i]$ with $m_1, \dots, m_N \geq 0$
be a non-zero element in $H_2(C,\BZ)$. Then for all $n \in \BZ$ the following holds.
\begin{enumerate}
\item[(i)] $\mathsf{N}_{\gamma,n}$ does not depend on the choice of polarization $H$.
\item[(ii)] $\mathsf{N}_{\gamma,n} = \mathsf{N}_{\gamma,-n}$.
\item[(iii)] Let $g = \mathrm{gcd}(m_1, \ldots, m_N)$. Then
$\mathsf{N}_{\gamma,n} = \mathsf{N}_{\gamma,n+g}$.
\item[(iv)] Assume there exist an $i$ such that $m_i > 0$ and the geometric genus of $C_i$ is non-zero.
Then $\mathsf{N}_{\gamma,n} = 0$.
\end{enumerate}
\end{lemma}
\begin{proof}
The independence of $H$ is \cite[Thm 6.16]{Joyce-Song}.
(ii) follows by an argument identical to \cite[Lemma 7.1]{Br1}. (iii) and (iv) is \cite[Section 2.9]{Tpar2}.
\end{proof}

\subsection{Rationality} \label{Section:Rationality}
Since $P_n(X,C)$ and therefore $\mathsf{PT}_C(q)$ only depends on the cycle class $[C] \in \Chow(X)$ we may write
$\mathsf{PT}_{[C]}(q) = \mathsf{PT}_C(q)$.

\begin{thm} \label{Theorem_Toda_equation}
The $C$-local Toda equation holds:
\begin{equation}
\sum_{\gamma \in H_2(C,\BZ)} \mathsf{PT}_{\gamma}(-q) v^{\gamma}
 = \exp\left( \sum_{\gamma > 0, n \geq 0} n \cdot \mathsf{N}_{\gamma,n} v^{\gamma} q^n \right) \cdot \sum_{\gamma \geq 0} L_{\gamma}(q),
\label{TodaEqn}
\end{equation}
where every $L_\gamma(q)$ is a Laurent polynomial invariant under $q \mapsto q^{-1}$.
\end{thm}

Toda's equation in combination with Lemma \ref{Lemma_Relations_NN} immediateley implies Theorem \ref{Theorem_Clocal}
by extracting the $v^{[C]}$ coefficient from \eqref{TodaEqn}.

\begin{proof}[Proof of Theorem \ref{Theorem_Toda_equation}]
If $\CM^{\text{ss}}$ is a moduli space
of semistable sheaves on $X$ of slope $m$ with respect to $\mu$,
then the pullback $\CM^{\text{ss}} \times_{\CA_C} \CA$ is set-theoretically the
moduli space of semistable sheaves of the same slope $m$ on $\CA_C$ with respect to $\mu_C$.
Hence, the claim follows from exactly the argument of Bridgeland in \cite[Section 7]{Br1}
by applying the ring homomorphism $\iota^{\ast}$ (resp. restricting to $\CA_C$) in every step.
\end{proof}

\subsection{Proof of Theorem \ref{Theorem_applications}}
Let $X = S \times E$ and let $(\beta,d) \in H_2(X, \BZ)$ be a curve class with $\beta \neq 0$.
The Hilbert-Chow morphism is equivariant
with respect to the translation by $E$
and hence descends to a map of (the underlying topological spaces of) quotient stacks
\[ \pi : P_n(X, (\beta,d)) / E  \to \Chow_{(\beta,d)}(X)/E \,. \]
By pushing-forward via $\pi$ we have
\[ \mathsf{N}_{n,(\beta,d)}^{X} =
\int_{P_n(X, (\beta,d)) / E} \nu \, \dd{e} \\
= \int_{\Chow_{(\beta,d)}(X) / E} \pi_{\ast}(\nu) \, \dd{e},
\]
where the constructible function $\pi_{\ast}(\nu)$ is defined by
\[ \pi_{\ast}(\nu)([C]) = \int_{\pi_{-1}([C])} \nu \dd{e}. \]
Since the (reduced) fiber of $\pi$ over every point
$[ C ]$ corresponding to a Cohen-Macaulay curve $C \subset X$
is $P_n(X,C)$, we have
\[ \pi_{\ast}(\nu)([C]) = \int_{P_n(X, C)} \nu \dd{e} = \mathsf{P}_{n,C}. \]
Hence Theorem \ref{Theorem_applications} (i) and (ii) follow directly from Theorem \ref{Theorem_Clocal} (i),(ii) by integration.

We prove (iii). By Corollary \ref{Corollary_Deformation_invariance}, we may take $\beta$ to be irreducible.
By the previous argument it is enough to show
\[
\mathsf{PT}_{[C]}(q) 
=
\sum_{n \in \BZ} \mathsf{P}_{n,C} q^n
=
\sum_{g=0}^{m} \mathsf{n}_{g, C} (q^{1/2} + q^{-1/2})^{2g-2}
\]
for some $\mathsf{n}_{g, C} \in \BQ$,
for every Cohen-Macaulay curve $C \subset X$ in class $(\beta,d)$.


Let $C_0, \dots, C_N$ be the irreducible components of $C$ where
$C_0$ is of class $(\beta,d_0)$ and the components $C_1, \dots, C_N$ have class $(0,1)$ in $H_2(X,\BZ)$.
Hence
\[ [C] = [C_0] + \sum_{i=1}^N d_i [C_i] \  \in H_2(C,\BZ) \]
where $\sum_{i=0}^{N} d_i = d$.
Since $C_1, \dots, C_N$ are of geometric genus $1$
the $C$-local DT invariant satisfies
\[ \mathsf{N}_{\Sigma_i e_i [C_i], n} = 0 \]
whenever $e_i > 0$ for some $i \geq 1$ by Lemma \ref{Lemma_Relations_NN}.
The $v^{[C]}$-coefficient of \eqref{TodaEqn} therefore reads
\[
\mathsf{PT}_{[C]}(-q) = \bigg( \sum_{n \geq 0} n \, \mathsf{N}_{[C_0], n} q^n \bigg) \cdot L_{[C]-[C_0]}(q) + L_{[C]}(q).
\]
Since $C_0$ has multiplicity $1$ we have by Lemma \ref{Lemma_Relations_NN}
\[
\sum_{n \geq 0} n \, \mathsf{N}_{[C_0], n} q^n = \sum_{n \geq 0} n \, \mathsf{N}_{[C_0], 1} q^n = \mathsf{N}_{[C_0],1} \frac{q}{(1-q)^2}.
\]
and therefore
\[
\mathsf{PT}_{[C]}(q) = \mathsf{N}_{[C_0],1} \frac{q}{(1+q)^2} \cdot L_{[C]-[C_0]}(-q) + L_{[C]}(-q).
\]
As Laurent polynomials invariant under $q \mapsto q^{-1}$
we may expand every $L_{\gamma}(-q)$ in terms of $(q^{1/2} + q^{-1/2})^{2 k}$ for $k \in \BZ_{\geq 0}$.
By integration we conclude that $\mathsf{PT}_{(\beta,d}(q)$ has the desired form with $\mathsf{n}_{g, (\beta,d)} \in \BQ$.
Since the $E$-action on the moduli space has no stabilizers for $\beta$ irreducible,
the series $\mathsf{PT}_{(\beta,d}(q)$ has integer coefficients.
By induction on $g$ we conclude $\mathsf{n}_{g, (\beta,d)} \in \BZ$. \qed

\appendix
\section{Fulton Chern class} \label{Appendix_Fulton_Chern_Class}
Let $X$ be a scheme with an embedding $\iota : X \to M$ into a non-singular ambient $M$.
The Fulton Chern class is defined by
\begin{equation} c_F(X) = c(T_{M|X}) \cap s(X,M) \label{FCC} \end{equation}
where $s(X,M)$ is the Segree class of $X$ in $M$.
The class \eqref{FCC} is independent of the choice of embedding \cite[4.2.6]{F}.

For lack of reference we give a proof of the following result.
\begin{lemma}
Let $X, X'$ be schemes which can be embedded into non-singular schemes,
and let $f : X \to X'$ be an \'etale morphism. Then 
\[ f^{\ast} c_F(X) = c_F(X') \,. \]
\end{lemma}
\begin{proof}
Let $\iota : X \to M$ and $\iota': X' \to M'$ be ambeddings into nonsingular schemes $M, M'$.
By replacing $\iota'$ with $(\iota \circ f, \iota'): X \to M' \times M$ we may assume there exist a smooth morphism $\rho : M' \to M$.
Let $Y = X \times_M M'$ be the fiber product. We have the commutative diagram
\[
\begin{tikzcd}
X' \ar{r}{j} \ar{dr}{f} & Y \ar{r}{\iota'} \ar{d}{\rho'} & M' \ar{d}{\rho} \\
              & X \ar{r}{\iota} & M
\end{tikzcd}
\]
where $j = (f, \iota')$ and $\rho'$ is the induced map. Since $f$ is \'etale and $\rho'$ is smooth,
$j$ is a regular embedding with normal bundle $N_{X'/Y} = j^{\ast} T_{\rho'} = j^{\ast} \iota^{\prime \ast} T_{\rho}$
where $T_{\rho}$ and $T_{\rho'}$ is the relative tangent bundle to $\rho$ and $\rho'$ respectively \cite[B.7.2]{F}.

Consider the composition $X' \xrightarrow{j} Y \xrightarrow{\iota'} M'$.
By an argument identical to \cite[4.2.6]{F} 
there exist an exact sequence $0 \to N_{X'/Y} \to C_{X'/M'} \to j^{\ast} C_{Y/M'} \to 0$
where $C_{X'/M'}$ denotes the cone of $X'$ in $M'$ and similar for $C_{Y/M'}$. Hence
we have $j^{\ast} s(C_{Y/M'}) = c(N_{X'/Y}) \cap s(X', M')$ \cite[4.1.6]{F}. We conclude
\begin{align*}
f^{\ast} c_F(X)
& = j^{\ast} \iota^{\prime \ast} \big( c(T_{M'}) / c(T_{\rho}) \big) \cap j^{\ast} \rho^{\prime \ast} s(X,M) \\
& = \big( c(T_{M'|X'}) \cdot c(N_{X'/Y})^{-1} \big) \cap j^{\ast} s(Y,M') \quad \quad (\text{since } \rho \text{ is flat}) \\
& = \big( c(T_{M'|X'}) \cdot c(N_{X'/Y})^{-1} \big) \cap c(N_{X'/Y}) \cap s(X', M') \\
& = c_F(X'). \qedhere
\end{align*}
\end{proof}

\end{document}